\documentclass[a4paper,11pt]{article}
\usepackage[utf8]{inputenc}
\usepackage{amsmath,amsthm,amssymb}
\usepackage{graphicx}
\usepackage[all]{xy}
\usepackage{verbatim}
\usepackage{enumerate}
\usepackage{float}
\usepackage{pgf}
\usepackage{tikz}
\usetikzlibrary{trees}
\usetikzlibrary{arrows,decorations.markings,patterns,calc,automata}

\theoremstyle{plain}
\newtheorem{theorem}{Theorem}[section]
\newtheorem{lemma}[theorem]{Lemma}
\newtheorem{corollary}[theorem]{Corollary}
\newtheorem{proposition}[theorem]{Proposition}
\theoremstyle{definition}
\newtheorem{remark}[theorem]{Remark}
\newtheorem{notation}[theorem]{Notation}
\newtheorem{definition}[theorem]{Definition}
\newtheorem{example}[theorem]{Example}
\newtheorem{question}[theorem]{Question}

\newcommand{\tree}{\mathcal T}
\newcommand{\vertices}{\mathcal V}
\newcommand{\edges}{\mathcal E}
\newcommand{\tfg}{\mathcal F}

\newcommand{\AutT}{\operatorname{Aut}(\mathcal T)}
\newcommand{\Aut}{\operatorname{Aut}}
\newcommand{\Fix}{\operatorname{Fix}}
\newcommand{\Sym}{\operatorname{Sym}}
\newcommand{\Alt}{\operatorname{Alt}}
\newcommand{\col}{\operatorname{col}}
\newcommand{\prm}{\operatorname{prm}}
\newcommand{\Homeo}{\operatorname{Homeo}}
\newcommand{\UF}{\mathcal{N}_F}
\newcommand{\sgn}{\operatorname{sgn}}

\newcommand{\AAut}{\operatorname{AAut}}
\newcommand{\frakg}{\mathfrak{g}}
\newcommand{\gpoid}{\mathcal G}

\usepackage[affil-it]{authblk}
\usepackage{titlefoot}

\begin{document}

\title{Coloured Neretin Groups}
\author{Waltraud Lederle \\ ETH Z\"urich}
\maketitle
\unmarkedfntext{This article was in 2019 accepted for publication in Groups, Geometry and Dynamics.}

\begin{abstract}
	We give sufficient conditions for a subgroup of a tree almost automorphism group to be isomorphic to the topological full groups of a one-sided shift in the sense of Matui.
	As an application, we show that almost automorphism groups of trees obtained from universal groups constructed by Burger and Mozes are compactly generated and virtually simple.
	In addition, using the approach of Bader, Caprace, Gelander and Mozes we show that some of these almost automorphism groups do not have any lattice.
\end{abstract}

\section{Introduction}

In the nineties Neretin \cite{n92} introduced a class of groups acting on the boundary of a regular tree by piecewise tree automorphisms. He thought of them as "combinatorial analogs of the group of diffeomorphism of the circle" and studied their representations. The groups are now known as Neretin's group and are mostly considered as groups of almost automorphisms of regular trees; see Section \ref{chdefalmaut} for definitions. They attracted the interest of group theorists when Kapoudjian \cite{k99} proved them to be simple. Equipped with a natural topology, Neretin's group is totally disconnected and locally compact, and it is now one of the fundamental and most interesting examples in the new growing structure theory of totally disconnected, locally compact groups, which has been mostly developed by Caprace, Reid and Willis; see for example \cite{crw17,crw17a}.
Caprace and De Medts \cite{cm11} showed that Neretin's group is compactly generated by showing it contains a dense copy of a Higman-Thompson group.
In fact, Neretin's group is even compactly presented; see Le Boudec~\cite{lb16}.
This result was strengthened by Sauer and Thumann \cite{st17}, 
who showed that it admits a cellular action on a contractible
cellular complex with compact open stabilizers and such that the restriction of the action on each $n$-skeleton is cocompact.

In his recent study of the topological full group of an \'etale groupoid, Matui \cite{m15} focused on groupoids associated to one-sided shifts of finite type. He showed that the aforementioned Higman-Thompson groups are examples of such topological full groups. This gives one class of examples of groups which can be realized as tree almost automorphism groups and also as a topological full groups associated to a one-sided shift of finite type. In the present work, see Theorem \ref{thmalmauttoshift}, we generalize this example in the following sense.
We will exhibit sufficient conditions determining when a subgroup of a tree almost automorphism group is isomorphic to the topological full group associated to a one-sided shift of finite type, and we will give the shift explicitly. This will allow us to explicitly determine their abelianization and prove compact generation for groups they embed in densely.

One reason why Neretin's group is of great interest is the result by Bader, Caprace, Gelander and Mozes \cite{bcgm12} that Neretin's group does not have any lattice.
Lattices play a tremendously important role in geometric group theory.
Let $\Lambda$ be a locally compact group. A \emph{lattice} $\Gamma$ in $\Lambda$ is a discrete subgroup such that there exists a finite, $\Lambda$-invariant measure on the quotient $\Lambda/\Gamma$.
Neretin's group was the first known example of a locally compact simple group not admitting any lattice.
Other groups having these properties, which are acting on trees, were constructed by Le Boudec \cite{lb16}.
Besides being interesting in itself, being simple in combination with having no lattice is a necessary condition to also not admit any nontrivial invariant random subgroup (IRS).
%
So far no example of a compactly generated, non-discrete, locally compact group without nontrivial IRS is known.
It is an open question whether Neretin's group has a nontrivial IRS.
A good introduction into IRSs are, for example, the notes of Gelander \cite{g15}.

The major part of the present paper ist devoted to the study of generalizations of Neretin's group, obtaining more examples of locally compact, compactly generated, simple groups without lattices. Let us shortly describe the construction.
Let $\tree$ be a regular tree of degree $d+1$ and denote by $\AutT$ its group of automorphisms with topology generated by all vertex stabilizers. Let $G \leq \AutT$ be a closed subgroup. We are interested in the group of homeomorphisms of the boundary $\partial\tree$ of $\tree$ consisting of all those homeomorphisms which "locally look like" elements of $G$. To make the last statement precise, one could say that $\tfg(G)$ is the topological full group
of $G$ acting on the boundary $\partial \tree$; see Section \ref{chdefalmaut} for precise definitions.
We prove that if $G$ has Tits' Independence Property, then there exists a unique group topology on this group such that the inclusion $G \hookrightarrow \tfg(G)$ is continuous and open; see Proposition \ref{proptopexists}.

We investigate more closely the case where $G$ is a universal group in the sense of Burger and Mozes.
For every vertex $v$ of $\tree$ we fix a bijection from the $d+1$ edges incident to $v$ to the set $D:=\{0,\dots,d\}$.
For every element $g \in \AutT$ we may, thus, talk about its local action at a vertex as an element of $\Sym(D)$.
Let $F \leq \Sym(D)$ be a subgroup of the symmetric group on $d+1$ letters.
Burger and Mozes \cite{bm00a} constructed closed subgroups $U(F) \leq \AutT$, called \emph{universal groups},
whose local action at every vertex is in $F$.
%

\begin{definition}
	A subgroup $F \leq \Sym(D)$ is called a \emph{Young subgroup} if there is no subgroup $F \lneq F' \leq \Sym(D)$ preserving the orbits of $F$.
\end{definition}

The following theorem summarizes the main results of this paper; for full statements and proofs see Corollary \ref{cor:dnfsple}, Theorem \ref{thm:dnffi}, Theorem \ref{cpctgenthm}, Theorem \ref{nolattice} and Theorem \ref{thmnococpctlattice}.

\begin{theorem} \label{thmintroduction}
	Let $F \leq \Sym(\{0,\dots,d\})$ be any subgroup. The following hold for $\UF:=\tfg(U(F))$.
	\begin{itemize}
		\item[a)] The commutator subgroup $D(\UF)$ of $\UF$ is open, simple and has finite index.
		More precisely, the abelianization $\UF/D(\UF)$ is a quotient of
		$(\mathbb{Z}/2\mathbb{Z})^{D/F}$.
		\item[b)] The group $\UF$ is compactly generated.
		\item[c)] If $F$ is a Young subgroup with strictly less than $d$ orbits, then $\UF$ does not have any lattice. If $F$ has precisely $d$ orbits, then $\UF$ does not have any cocompact lattices.
	\end{itemize}
	Consequently, if $F$ is a Young subgroup with less than $d$ orbits, then $D(\UF)$ is a compactly generated, non-discrete, simple group without lattices.
\end{theorem}

The first two statements are an application of the connection between almost automorphism groups of trees and topological full groups associated to one-sided shifts of finite type. This connection allows us to find a dense subgroup of $\UF$ that, by Matui's work, is finitely generated.
This is the generalization of the fact that Neretin's group contains dense copies of a Higman-Thompson group.
The proof of the third statement follows the approach of Bader, Caprace, Gelander and Mozes~\cite{bcgm12}.

We want to point out the following questions that are left open in this article and we are interested in.

\begin{question}
	If a subgroup $G \leq \AutT$ has Tit's Independence Property and acts minimally on the tree, is its group of almost automorphisms $\tfg(G)$ compactly generated? Does it always densely contain a topological full group coming from a one-sided shift in the sense of Matui?
\end{question}

\begin{question}
	Let $\tree$ be a regular tree. If $G \leq \AutT$ is a closed subgroup satisfying Tit's Independence Property, which properties distinguish $\tfg(G)$ from Neretin's group?
\end{question}

\begin{question}
	It appears rather arbitrary that the proof that $\UF$ does not have any lattices goes through only if $F$ is a Young subgroup. If $F$ is not a Young subgroup, does $\UF$ have lattices?
	For transitive $F$ satisfying specific additional properties, the answer can be found in the work by Le Boudec \cite{lb16}, Corollary 7.7.. He constructs examples of groups without lattices which embed as open subgroups into $\UF$. For general $F$ the question whether $\UF$ has lattices or not is still open.
	Especially interesting would, of course, be a proof following a different approach to the one by Bader, Caprace, Gelander and Mozes or the one by Le Boudec.
\end{question}

\begin{question}
	Does $\UF$ have invariant random subgroups?
\end{question}

\subsection{Organization of the paper}

In Section 2 we set up basic notations, definitions and terminology.
We introduce universal groups, almost automorphism groups of trees and Higman-Thompson groups.
We establish some basic results about topological full groups of subgroups of $\AutT$ for a regular tree $\tree$.
We also define the group $\UF$ and show the existence of a locally compact group topology on
a class of almost automorphism groups including $\UF$.

In Section 3 we introduce the topological full group of an \'etale groupoid following Matui.
We also define one-sided irreducible shifts of finite type.
We give a connection between topological full groups associated to one-sided shifts and tree almost automorphisms.
As an application we will find a finitely generated subgroup $V_F \leq \UF$ which we think of as an analog of the Higman-Thompson group.

In Section 4 we show that $V_F$ is dense in $\UF$ and conclude that $\UF$ is compactly generated and its commutator subgroup is open, simple and has finite index. We also give normal subgroups of $\UF$.

In Section 5 we prove the third part of Theorem \ref{thmintroduction}.

Section 5 does not rely on Sections 3 and 4 and can be read independently.

\subsection{Acknowledgments}

I want to thank Marc Burger for helpful discussions and comments on previous versions of this work, Mario \v{S}iki\'c for careful reading of early versions and help with computer
algebra systems for the lattice part, David Robertson for fruitful discussions, and Nicol\'as Matte Bon, Adrien Le Boudec, Pierre-Emmanuel Caprace and Stephan Tornier for useful remarks about previous versions of this article.
I am particularly grateful to the anonymous referee for a very careful reading, plenty of good suggestions and pointing out a mistake in an earlier version.
I also want to acknowledge the hospitality of the Australian Winter of Disconnectedness in Creswick, Victoria and Newcastle, New South Wales, where part of this work was completed.

\section{Preliminaries and basic results}

\subsection{Trees}

In this subsection we establish notations and conventions about graphs and trees.
We follow the notion of Serre \cite{s03}. In particular, an edge $e$ has an origin $o(e)$, a terminus $t(e)$ and an inverse edge $\overline{e}$.

In the whole paper $\mathcal T = (\mathcal V,\mathcal E)$ denotes a locally finite tree with vertex set $\mathcal V$ and edge set $\mathcal E$. We assume it has no leaves and no isolated points in the boundary (see below for the definition of the boundary).
Endow $\tree$ with the usual metric such that for all $v,w \in \mathcal{V}$ 
the distance between $v$ and $w$ is the length of the geodesic (i.e. the shortest path) from $v$ to $w$.
Fix a vertex $v_0 \in \mathcal V$ of $\tree$ and consider $\tree$ as a rooted tree with root $v_0$. 
Now we can talk about the parent and the $d$ children of a vertex of $\tree$, namely its neighbours closer respectively more distant to $v_0$ (only $v_0$ does not have a parent but $d+1$ children).
A path starting at $v_0$ is called \emph{rooted}.
%

\begin{definition}
	The \emph{boundary} of $\tree$ is the set of all rooted infinite geodesics in $\tree$.
	It is denoted by $\partial \tree$.
\end{definition}

\paragraph{Topology on $\partial \tree$.}
For every vertex $v \in \mathcal V$ we denote by $\tree_v$ the subtree of $\tree$ 
whose vertices are all $w \in \mathcal V$ such that $v$ lies on the rooted geodesic to $w$.
It is a rooted tree with root $v$.
Its boundary $\partial \tree_v$ is a subset of $\partial \tree$ in an obvious way.
The set $\{\partial \tree_v \mid v \in \mathcal V\}$ is a basis of the topology on $\partial \tree$. 
With this topology $\partial \tree$ is a Cantor space.

\begin{figure}[H]
	\centering
	\begin{tikzpicture}[font=\footnotesize]  [grow'=down]
	
	\tikzstyle{level 1}=[level distance=5mm,sibling distance=20mm]
	\tikzstyle{level 2}=[level distance=5mm,sibling distance=10mm]
	\tikzstyle{level 3}=[level distance=5mm,sibling distance=5.5mm]
	\tikzstyle{level 4}=[level distance=5mm,sibling distance=2.5mm]
	\tikzstyle{level 5}=[level distance=4mm,sibling distance=1.5mm]
	
	\node [inner sep=0pt,outer sep=-0.4pt] at (0,0) {}
	child
	{child {child {child  child  } child {child  child  } } 
		child {child {child  child  } child {child  child  } } }
	child
	{child {child {child  child  } child {child  child  } } 
		child {child[very thick] {child  child  } child[ultra thick] {child  child  } } }
	child
	{ child {child {child  child  } child {child  child  } } 
		child {child {child  child  } child {child  child  } } };
	
	\node[above] at(0,0){$v_0$};
	\node[above]at(0.5,-1){$v$};
	\end{tikzpicture}
	\caption{The thick lines indicate the subtree $\tree_v$.}
\end{figure}
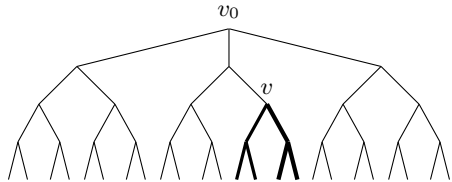
%

\paragraph{Automorphisms of $\tree$.}
Denote by $\AutT$ the group of automorphisms of $\tree$, that is, all graph morphisms $\tree \to \tree$ which are
bijective on $\mathcal{V}$ and $\mathcal{E}$. We define a group topology on $\AutT$ making it into a totally disconnected locally compact group.
Let $G \leq \AutT$. Denote by $\Fix_G(L) \leq G$ all the elements of $G$ which fix every element of $L$.
A neighbourhood basis of the identity in $\AutT$ consists of all subgroups of the form $\Fix_{\AutT}(L)$ with $L \subset \mathcal V \cup \mathcal E$ finite. With this topology each of these basis elements is compact and open.

\paragraph*{Terminology.}
A subset of a topological space is called \emph{clopen} if it is closed and open.

\paragraph{Tits' Independence Property.}
Tits \cite{t70}, Section 4.2, defined a property for subgroups of $\AutT$
and proved a simplicity theorem for groups satisfying it.
Let $L$ be any path in $\tree$.
For every vertex $v \in \mathcal{V}$ denote by $\pi(v)$ the unique vertex of $L$ which is closest to $v$.
For such a vertex $w$ let $L_w$ be the subtree spanned by $\pi^{-1}(w)$, that is, the inclusion-minimal subtree of $\tree$ containing $\pi^{-1}(w)$.
Let $G \leq \AutT$. The group $\Fix_G(L)$ leaves $L_w$ invariant.
Thus, for every $g \in \Fix_G(L)$ the restriction $|_{L_w} \colon \Fix_G(L) \to \Aut(L_w)$ is a well-defined homomorphism.
The group $G$ is said to have \emph{Tits' Independence Property} if for every $L$ as above the induced map $\Fix_G(L) \to \prod_{w}\Fix_G(L)|_{L_w}$ is an isomorphism.

\begin{remark} \label{titsrem}
	If $G \leq \AutT$ is closed, Tits' Independence Property is equivalent to each of the following conditions,
	which for non-closed $G$ are weaker, but equivalent to each other, in general (see \cite{a03}, Section 1.2.):
	\begin{itemize}
		\item replacing ``any path $L$'' by ``any finite subtree $L$ of $\tree$'';
		\item replacing ``any path $L$'' by ``any edge $L$ of $\tree$''.
	\end{itemize}
\end{remark}

The importance of Tits' Independence Property lies in the following theorem.

\begin{theorem}[\cite{t70}, Theorem 4.5] \label{thmtitssimpl}
	Let $G \leq \AutT$ be a subgroup satisfying Tits' Independence Property. Assume that $G$ neither preserves any proper subtree nor fixes any element of $\partial\tree$. Then, the subgroup $$G^+ := \langle \{\Fix_G(e) \mid e \in \mathcal{E} \} \rangle \leq G$$ generated by all edge fixators in $G$ is simple or trivial.
\end{theorem}

\subsection{Colorings and universal groups} \label{burgermozesgroups}

In this subsection $\tree=(\vertices,\edges)$ is a $(d+1)$-regular tree.
Definitions and statements presented here are, unless otherwise stated, due to Burger and Mozes \cite{bm00a}, Section 3.2.
For a more detailed introduction and proofs we refer to \cite{ggt18}, Section 4.

\paragraph{Legal colourings.}
Throughout the paper we denote $D := \{0,\dots,d\}$ and call it the \emph{set of colours}. 
We fix a \emph{legal edge colouring} of $\tree$, that is a map $$\operatorname{col} \colon \mathcal E \to D$$
satisfying the following two properties.
\begin{itemize}
	\item It is constant on geometric edges, i.e. $\col(e)=\col(\bar e)$ for all $e \in \mathcal E$.
	\item For every $v \in \mathcal V$ the edges incident to $v$ all have different colours, i.e.
	the restriction $\col|_{o^{-1}(v)} \colon o^{-1}(v) \to D$ is a bijection.
\end{itemize}

Let $F \leq \operatorname{Sym}(D)$ be any subgroup.
Every automorphism $g \in \AutT$ induces for each vertex $v \in \mathcal V$ a permutation $\prm_{g,v}\in \operatorname{Sym}(D)$ defined by $$\prm_{g,v}(\chi) = \operatorname{col}(g((\operatorname{col}|_{o^{-1}(v)})^{-1}(\chi))).$$

\begin{definition}
	The \emph{universal group} associated to $F$ is defined by
	\[
	U(F) = \{g \in \AutT \mid \forall v \in \mathcal V\colon \prm_{g,v} \in F\}.
	\]
	Informally speaking, it consists of all tree automorphisms whose local action is everywhere prescribed by $F$.
\end{definition}

That this indeed defines a group is due to the following lemma.

\begin{lemma}[\cite{ggt18}, Lemma 4.2.] \label{lemmultprm}
	Let $g,h \in U(F)$ and $v \in \mathcal{V}$.
	Then $\prm_{gh,v}=\prm_{g,hv} \circ \prm_{h,v}$.
\end{lemma}

\begin{remark}
	A different choice of a legal colouring will result in a universal group that is conjugate to the original one.
\end{remark}

\begin{remark}\label{rmkufplustriv}
	For every $F \leq \Sym(D)$ the universal group $U(F)$ is a closed subgroup of $\AutT$ satisfying Tits' Independence Property.
	It is not hard to see that the group $U(F)$ is discrete if and only if the action $F \curvearrowright D$ is free, which is again equivalent to $U(F)^+=\{1\}$.
\end{remark}

The following lemma about extending certain tree automorphisms to ``almost being in $U(F)$'' was formulated by Le Boudec in the case of a ball around a vertex.
A close look at the proof shows that it is valid for every subtree of $\tree$.

\begin{lemma}[\cite{lb16}, Lemma 3.4.] \label{lemextuf}
	Let $T$ be a subtree of $\tree$.
	Let $h \in \AutT$ be such that for every vertex $v$ of $T$ the permutation
	$\prm_{h,v}$ preserves the orbits of $F$.
	Then there exists $g \in \AutT$ such that $g|_T=h|_T$ and such that
	for all vertices $w \in \mathcal{V}$
	which are either leaves of $T$ or not vertices of $T$ holds $\prm_{g,v} \in F$.
\end{lemma}

%
%

\subsection{Almost automorphisms} \label{chdefalmaut}

\begin{definition}
	A finite subtree $T \subset \tree$ is called \emph{complete} if it contains the root $v_0$ and if for every vertex $v$ of $T$ that is not a leaf all children of $v$ are also in $T$.
\end{definition}

\begin{notation}
	For a subtree $T \subset \tree$ we will denote by $\mathcal{L}T \subset \mathcal V$ the set of leaves of $T$.
\end{notation}

\begin{notation}
	For a finite complete subtree $T \subset \tree$ the difference $\tree \setminus T$ will always denote the subgraph $\bigsqcup_{v \in \mathcal{L}T} \tree_v \subset \tree$.
	Hence $\tree \setminus T$ is a forest with $|\mathcal{L}T|$ many connected components.
\end{notation}

\begin{definition}
	Let $T_1$ and $T_2$ be finite complete subtrees of $\tree$. 
	An \emph{honest almost automorphism} of $\tree$ is a forest isomorphism $\varphi \colon \tree \setminus T_1 \to \tree \setminus T_2$. 
\end{definition}

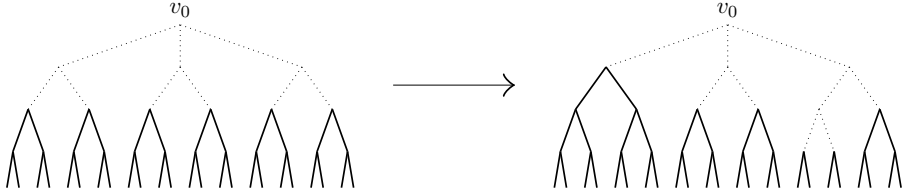
\begin{figure}[H]
	\centering
	\scalebox{0.8}{
		\begin{tikzpicture} [grow'=down]
		
		\tikzstyle{level 1}=[level distance=7mm,sibling distance=20mm]
		\tikzstyle{level 2}=[level distance=7mm,sibling distance=10mm]
		\tikzstyle{level 3}=[level distance=7mm,sibling distance=5mm]
		\tikzstyle{level 4}=[level distance=6mm,sibling distance=2mm]
		\tikzstyle{level 5}=[level distance=5mm,sibling distance=1mm]
		
		\node [inner sep=0pt,outer sep=-0.4pt] at (-4.5,0) {}
		child[dotted] foreach \i in {1,2,3}
		{child {child[solid,thick] {child  child  } 
				child[solid,thick] {child  child  } } 
			child {child[solid,thick] {child  child  } 
				child[solid,thick]  {child child  } } }; 
		\node [inner sep=0pt,outer sep=-0.4pt] at (4.5,0) {}
		child[dotted]
		{child {child[solid,thick] {child  child  } child[solid,thick] {child  child  } } 
			child {child {child[solid,thick]  child[solid,thick]  } child {child[solid,thick]  child[solid,thick]  } } }
		child[dotted]
		{child {child[solid,thick] {child  child  } child[solid,thick] {child  child  } } 
			child {child[solid,thick] {child  child  } child[solid,thick] {child  child  } } }
		child[dotted]
		{ child[solid,thick] {child {child  child  } child {child  child  } } 
			child[solid,thick] {child {child  child  } child {child  child  } } };
		
		\node [above] at (-4.5,0) {$v_0$};
		\node [above] at (4.5,0) {$v_0$};
		;
		
		\draw[decoration={markings,mark=at position 1 with {\arrow[line width=0.1mm,scale=3.5]{>}}},postaction={decorate}] (-1,-1)--(1,-1);    
		\end{tikzpicture}
	}
	\caption{The trees $T_1$ and $T_2$ are indicated with dotted lines.}
	\label{almautfigure}
\end{figure}

\paragraph{Almost automorphisms.} We now construct an equivalence relation on the set of honest almost automorphisms. 
Let $T_1,T_2,T'_1,T'_2 \subset \tree$ be finite complete subtrees of $\tree$. Let
$\varphi \colon  \tree \setminus T_1 \to \tree \setminus T_2$ and $\psi \colon  \tree \setminus T'_1 \to \tree \setminus T'_2$ be honest almost automorphisms of $\tree$. 
We say that $\varphi$ and $\psi$ are \emph{equivalent} if there exists a finite complete subtree $T \supset T_1 \cup T'_1$
such that $\varphi|_{\tree \setminus T} = \psi|_{\tree \setminus T}$.

An \emph{almost automorphism} of $\tree$ is the equivalence class of an honest almost automorphism under this equivalence relation. In our notation we will usually not distinguish between an honest almost automorphisms and its equivalence class,
but say it explicitly whenever we need to talk about an honest almost automorphism.

\paragraph{Simple expansions.}
In proofs it will be convenient to work with generators for this equivalence relation.
For finite complete subtrees $T \subset T' \subset \tree$, we say that $T'$ is obtained from $T$ by a \emph{simple expansion} if there exists a leaf $v$ of $T$ such that $T'$ is spanned by $T$ and the children of $v$. Note that any finite complete subtree of $\tree$ containing $T$ is obtained from $T$ by a sequence of simple expansions.
If in the preceding paragraph we require that $T'_1$ is obtained from $T_1$ by a simple expansion and $T=T'_1$,
the resulting relation generates the equivalence relation. 

\begin{remark} \label{equivclassrmk}
	Let $T_1,T_2$ be finite subtrees with the same number of leaves and let $\varphi \colon \tree \setminus T_1 \to \tree \setminus T_2$ be an honest almost automorphism.
	Then, for every finite complete subtree $T \subset \tree$ containing $T_1$ there exists a unique finite complete subtree $T' \subset \tree$ containing $T_2$ and a unique representative $\psi \colon \tree \setminus T \to \tree \setminus T'$ of $\varphi$.
	Explicitly $T' = \varphi(T \setminus T_1) \cup T_2$ and $\psi = \varphi|_{\tree \setminus T}$.
	
	The analogous statement holds for $T \supset T_2$.
\end{remark}

\paragraph{Product of two almost automorphisms.} Take finite complete subtrees $T_1,\dots,T_4 \subset \tree$.
Let $\varphi \colon \tree \setminus T_1 \to \tree \setminus T_2$ and $\psi \colon  \tree \setminus T_3 \to \tree \setminus T_4$ be almost automorphisms.
By the previous remark we can choose a finite complete subtree $T \supset T_4 \cup T_1$ of $\tree$ and take representatives for $\psi$ and $\varphi$ with image respectively domain $\tree \setminus T$. These representatives we can compose. The equivalence class of this composition is the product $\varphi \circ \psi$.
With this product the set of almost automorphisms of $\tree$ is a group, the \emph{almost automorphism group of the tree}, denoted $\AAut(\tree)$.

\begin{definition}
	If $\tree$ is a regular tree, then $\AAut(\tree)$ is called \emph{Neretin's group}.
\end{definition}

\paragraph{The group of almost automorphisms for a subgroup of $\AutT$.} 
Let $G \leq \AutT$.
We define its group $\mathcal{F}(G)$ of almost automorphisms. 
Let $T_1,T_2 \subset \tree$ be finite complete subtrees. 
A $G$-\emph{honest almost automorphism} of $\tree$ is an honest almost automorphism
$\varphi \colon \tree \setminus T_1 \to \tree \setminus T_2$ such that
for every $v \in \mathcal{L}T_1$ there exists a $g_v \in G$ with $\varphi|_{\tree_v}=g_v|_{\tree_v}$.
The elements of $\mathcal{F}(G)$ are the equivalence classes of all $G$-honest almost automorphisms.
It is not hard to see that $\mathcal{F}(G)$ is a subgroup of $\AAut(\tree)$.

\begin{remark}
	Note that Remark \ref{equivclassrmk} remains true for $G$-honest almost automorphisms.
\end{remark}

\begin{remark}
	Let $T_1$ and $T_2$ be finite complete subtrees of $\tree$ and
	let $G \leq \AutT$ be a subgroup.
	It is possible that there does not exist any $G$-honest almost automorphism
	$\varphi \colon \tree \setminus T_1 \to \tree \setminus T_2$.
	Consider for example the group $G=\Fix_{\AutT}(v_0)$.
	Then there is no $G$-honest almost automorphism as indicated in Figure \ref{almautfigure}
	because clearly every $G$-honest almost automorphism $\varphi \colon \tree \setminus T_1 \to \tree \setminus T_2$ needs to preserve the distance of the leaves of $T_1$ to $v_0$.
\end{remark}

\paragraph*{The intersection $\mathcal{F}(G) \cap \AutT$.} For a subgroup $G \leq \AutT$, sometimes the intersection $\mathcal{F}(G) \cap \AutT$ is of interest.
In general it is strictly larger than $G$ and can be much different.
For example, even if $G$ is a closed subgroup of $\AutT$, this is in general not the case for $\mathcal{F}(G) \cap \AutT$.
It is easy to see that $\mathcal{F}(G) \cap \AutT$ enjoys a weaker form of Tits' independence property,
where $L$ is replaced by arbitrary finite subtrees, see Remark \ref{titsrem}.
Le Boudec investigated the intersection for regular trees $\tree$ and $G=U(F)$ in \cite{lb16} 
and for more general $G \leq \Aut(\tree)$ in \cite{lb17a}, Section 4.

We now generalize a proposition by Le Boudec, Lemma 3.3 in \cite{lb16}.
Recall that by convention all our trees are locally finite.

\begin{proposition} \label{lemsameorb}
	Let $\tree$ be a tree without leaves.
	Let $G \leq \AutT$ be any subgroup. The orbits of $G$ on the directed edges of $\tree$ are the same as the orbits of $\mathcal{F}(G)\cap \AutT$.
\end{proposition}

\begin{proof}
	It suffices to show that $G$ and $\mathcal{F}(G)\cap \AutT$ have the same orbits on edges.
	Recall that an edge by definition has an origin and a terminus, i.e. a tree automorphism fixing an edge also fixes both of its endpoints.
	It is obvious that every orbit of $G$ is contained in an orbit of $\mathcal{F}(G)\cap \AutT$.
	
	Let, by contradiction, $h \in \mathcal{F}(G)\cap \AutT$ be such that there exists an edge $e$ of $\tree$ with $h(e) \notin Ge$. We call such an edge a \emph{bad edge} for $h$.
	Note that $h(\bar e) = \overline{h(e)}$ and therefore the inverse edge of a bad edge is also bad.
	Since $h$ coincides with automorphisms of $\tree$ on all but finitely many edges of $\tree$, it has at most finitely many bad edges. Therefore, there exists a bad edge with no bad children, we assume $e$ is such and we denote $v=o(e)$.
	Formally, this implies that
	for all edges $e'$ with $o(e')=v$ and $e' \neq e$ there exists a $g_{e'} \in G$ with $g_{e'}(e')=h(e')$. Note that $v$ is not a leaf, i.e. it has children.
	
	\begin{figure}[H]
		\centering
		\scalebox{0.7}{
			\begin{tikzpicture} [grow'=down]
			
			\tikzstyle{level 1}=[level distance=20mm,sibling distance=40mm]
			\tikzstyle{level 2}=[level distance=15mm,sibling distance=20mm]
			\tikzstyle{level 3}=[level distance=7mm,sibling distance=5mm]
			\tikzstyle{level 4}=[level distance=6mm,sibling distance=2mm]
			\tikzstyle{level 5}=[level distance=5mm,sibling distance=1mm]
			
			\node [inner sep=0pt,outer sep=-0.4pt] at (-4.5,0) {}
			child {node {$v$} {child{edge from parent [->] node [above] {$e'$}}
					child{edge from parent [->]}
					child{edge from parent [->]}}edge from parent [<-]  node [right] {$e$}}; 
			\node [inner sep=0pt,outer sep=-0.4pt] at (4.5,0) {}
			child {node {$h(v)$} {child{edge from parent [->] node [right] {$h(e')=g_{e'}(e')$}}
					child{edge from parent [->]}
					child{edge from parent [->]}}edge from parent [<-]  node [right] {$h(e)$}}; 
			\draw[decoration={markings,mark=at position 1 with {\arrow[line width=0.1mm,scale=2]{>}}},postaction={decorate}] (-1,-1.5)--(1,-1.5);    
			\end{tikzpicture}
		}
		\caption{Illustration of $o^{-1}(v)$ and $o^{-1}(h(v))$}
		\label{figov}
	\end{figure}
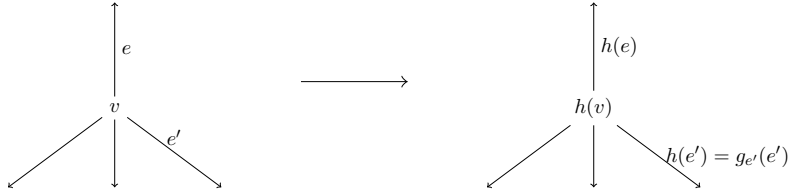
	
	Note that we do not assume that $h|_{\tree_{t(e')}}=g_{e'}|_{\tree_{t(e')}}$.
	Since there exists an element $g \in G$ such that $g(v)=h(v)$, for example $g_{e'}$ for $e' \neq e$,
	the subgroups $\Fix_G(v)$ and $\Fix_G(h(v))$ of $G$ are conjugated and act orbit-equivalently on $o^{-1}(v)$ and $o^{-1}(h(v))$, namely for example via $g_{e'}$.
	In particular, for all $e' \neq e$ holds $$|\Fix_G(v)\cdot e'|=|\Fix_G(h(v))\cdot h(e')|.$$
	Let now $e_1 \neq e$ be an element of $o^{-1}(v)$.
	Since $g_{e_1}(e) \neq h(e)$ there exists an $e_2 \in o^{-1}(v)$ with $g_{e_2}(e_2)=g_{e_1}(e)$.
	%
	In particular $e$ and $e_2$ are in the same $\Fix_G(v)$-orbit.
	Now if $h(e)$ and $h(e_2)$ are in the same $\Fix_G(h(v))$-orbit, we find an element $g \in G$ such that $g(e)=h(e)$, contradiction. Hence $h(e)$ and $h(e_2)$ are not in the same orbit.
	But since the orbits of $e_2$ and $h(e_2)$ need to have the same cardinalities,
	there exists an $e_3$ not equal to $e$ and $e_2$ such that $h(e_3)$ and $h(e_2)$ are in the same $\Fix_G(h(v))$-orbit, but $e_3$ is not in $\Fix_G(v)\cdot e_2$.
	That means there exists a $g \in G$ such that $g(e_3)=h(e_2)$.
	But then $g_{e_2}^{-1}g(e_3)=e_2$. This is a contradiction.
\end{proof}
%

\paragraph{Remark/Warning.}
Let $G \leq \AutT$ be a subgroup. Consider an honest almost automorphism
$\varphi \colon \tree \setminus T_1 \to \tree \setminus T_2$ for finite complete subtrees $T_1,T_2$ of $\tree$
such that the equivalence class of $\varphi$ is an element of $\mathcal{F}(G)$. 
It is in general not true that $\varphi$ is a $G$-honest almost automorphism.

\paragraph{The group $\mathcal{F}(G)$ as topological full group.}
We now give an alternative description of $\mathcal{F}(G)$ and prove that it is equivalent to the previous one.

\begin{definition}\label{deftopfull}
	Consider a group $\Lambda$ acting on a topological space $X$. 
	The \emph{topological full group} of this action is the following subgroup of
	$\Homeo(X)$. It consists of all those homeomorphisms $\varphi \colon X \to X$ such that for every $x \in X$ there exists a neighbourhood $U$ of $x$ and a group element $g \in \Lambda$ with $\varphi|_U = g|_U$. 
\end{definition}



\begin{lemma}
	For a subgroup $G \leq \AutT$ the group $\mathcal{F}(G)$ is isomorphic to the topological full group of $G$ acting on $\partial \tree$. Consequently, it acts itself faithfully on $\partial \tree$.
\end{lemma}

\begin{proof}
	The group $\mathcal{F}(G)$ can be seen as a subgroup of $\operatorname{Homeo}(\partial \tree)$ in the same way as $G$.
	Let $T_1,T_2$ be finite complete subtrees of $\tree$ with the same number of leaves and let
	$\varphi \colon \tree \setminus T_1 \to \tree \setminus T_2$ be a $G$-honest almost automorphism of $\tree$.
	Let $x \in \partial \tree$. Then $x$ is a path starting at $v_0$ and it needs to pass through one of the leaves of $T_1$. 
	Let $v \in \mathcal{L}T_1$ be this leaf. There exists a $g_v \in G$ with $\varphi|_{\tree_v}=g_v|_{\tree_v}$.
	For the action of $\varphi$ on $\partial \tree$ this means $\varphi|_{\partial \tree_v}=g_v|_{\partial \tree_v}$, and clearly $\partial \tree_v$ is an open neighbourhood of $x$.
	Therefore the equivalence class of $\varphi$ is an element of the topological full group of $G$ acting on $\partial \tree$.
	
	Let on the other hand $\psi$ be an element of the topological full group of $G$ acting on $\partial \tree$.
	By definition of the topology on $\partial \tree$ and by compactness there exist finitely man vertices
	$v_1,\dots,v_n \in \mathcal{V}$ and $g_1,\dots,g_n \in G$ such that 
	$\partial \tree = \bigsqcup_{i=1}^n \partial \tree_{v_i}$ and such that
	for every $i$ holds
	$\psi|_{\partial \tree_{v_i}} = g_i|_{\partial \tree_{v_i}}$.
	We claim that the subtree $T$ of $\tree$ spanned by $\{v_1,\dots,v_n\}$ is finite and complete.
	It is the union of all images of rooted geodesics to $v_i$ for $i=1,\dots,n$.
	Finiteness is therefore obvious.
	For completeness, assume that there is a vertex $v$ of $T$ which is not a leaf, but such that $v$ has a child $w$ which
	is not a vertex in $\tree$. Then $\partial \tree_w$ is not contained in $\bigcup_{i=1}^n \partial \tree_{v_i}$,
	which is a contradiction.
	With the same argument the tree $T'$ spanned by $\{g_1(v_1),\dots,g_n(v_n)\}$ is finite and complete and
	we can see that $\psi$ can be viewed as a $G$-honest almost automorphism
	$\tree \setminus T \to \tree \setminus T'$ with $\psi|_{\tree_{v_i}}=g|_{\tree_{v_i}}$.
\end{proof}

\subsubsection{Topology on almost automorphisms}
Let $G \leq \AutT$ be a closed subgroup. In particular $G$ is totally disconnected and locally compact.
We want to define a group topology on $\mathcal{F}(G)$ such that $G \leq \mathcal{F}(G)$ is open. This is not always possible, but we will see that we can do it if $G$ has Tits' Independence Property. First we need a lemma from Bourbaki.

\begin{lemma}[\cite{b98}, Ch. III, Sect. I, Subsect. 2, Prop. 1] \label{bourbakithm}
	Let $\Lambda$ be a group and $\mathcal B$ be a filter on $\Lambda$ satisfying the following three conditions.
	\begin{enumerate}
		\item For every $U \in \mathcal B$ there exists a $V \in \mathcal B$ such that $VV \subset U$.
		\item For every $U \in \mathcal B$ holds $U^{-1} \in \mathcal B$.
		\item For every $g \in \Lambda$ and every $V \in \mathcal B$ holds $gVg^{-1} \in \mathcal B$.
	\end{enumerate}
	Then, there exists a unique group topology on $\Lambda$ such that $\mathcal B$ is a neighbourhood basis of the identity element.
\end{lemma}

Recall the important theorem of Van Dantzig about totally disconnected locally compact groups.

\begin{theorem}[\cite{d36}, TG. 39]
	For every totally disconnected locally compact group the set of its compact open subgroups is a neighbourhood basis of the identity.
\end{theorem}

\begin{proposition} \label{proptopexists}
	Assume $G \leq \AutT$ is closed and has Tits' Independence Property. Then there exists a unique group topology on $\mathcal{F}(G)$ such that
	$G \leq \mathcal{F}(G)$ is an open subgroup.
\end{proposition}
%

\begin{proof}
	Since $G$ is totally disconnected and locally compact, any group topology on $\mathcal{F}(G)$ for which $G \leq \mathcal{F}(G)$ is open also has to be totally disconnected and locally compact.
	So we choose the filter on $\mathcal{F}(G)$ defined by 
	$$\mathcal B = \{U \subset \mathcal{F}(G) \mid U \cap G \subset G \text{ contains an open neighbourhood of }id\}.$$ 
	Conditions 1 and 2 of Lemma \ref{bourbakithm} are clearly fulfilled by van Dantzig's Theorem.
	
	To verify Condition 3 let $\varphi \in \mathcal{F}(G)$ and $O \in \mathcal{B}$ be arbitrary.
	Take a finite complete subtree $T \subset \tree$ such that $\operatorname{Fix}_{G}(T) \subset O$.
	By Remark \ref{equivclassrmk} we can choose $T$ big enough such that 
	there exists a finite complete subtree $T' \subset \tree$ and a representative
	$\varphi \colon \tree \setminus T' \to \tree \setminus T$ as $G$-honest almost automorphism.
	Let $g \in \operatorname{Fix}_{G}(T')$ be arbitrary.
	Then the $G$-honest almost automorphism $\varphi g \varphi^{-1} \colon \tree \setminus T \to \tree \setminus T$ fixes the leaves of $T$ and
	therefore extends to an element in $\Fix_{\mathcal{F}(G) \cap \AutT}(T)$. Also, on each connected component of $\tree \setminus T$ it coincides with an element of $G$. Hence by Tits' Independence Property 
	it is an element of $\operatorname{Fix}_{G}(T) \subset O$.
	We thus proved that $\varphi \operatorname{Fix}_{G}(T') \varphi^{-1} \subset O$, or equivalently
	$ \operatorname{Fix}_{G}(T')\subset  \varphi^{-1} O\varphi$, which shows $\varphi^{-1}O\varphi \in \mathcal B$.
\end{proof}

It is obvious that whenever $H$ is a subgroup of $G \leq \AutT$, then also $\mathcal{F}(H)$ is a subgroup of $\mathcal{F}(G)$. The next proposition relates the actions of $H$ and $G$ on $\tree$ to a topological property of the almost automorphism groups.

\begin{proposition} \label{propdense}
	Let $H \leq G$ be subgroups of $\AutT$. Assume $G$ is closed in $\AutT$ and has Tits' independence property,
	and endow $\mathcal{F}(G)$ with the topology from Proposition \ref{proptopexists}.
	Then $\mathcal{F}(H)$ is dense in $\mathcal{F}(G)$ if and only if $H$ and $G$ have the same orbits on directed edges of $\tree$.
\end{proposition}

\begin{proof}
	%
	Assume $\overline{\mathcal{F}(H)} = \mathcal{F}(G)$.
	Let $e$ be an edge of $\tree$ and let $g \in G$.
	Let $T \subset \tree$ be a finite complete subtree containing $g(e)$. 
	Since $\mathcal{F}(H)$ is dense in $\mathcal{F}(G)$ and $\operatorname{Fix}_{G}(T)$ is open, there exist $\psi \in \mathcal{F}(H)$ and $g' \in \operatorname{Fix}_{G}(T)$
	such that $g = g' \psi$.
	Hence, $\psi = g'^{-1} g$ is an element of $\mathcal{F}(H)\cap \AutT$ satisfying $\psi(e)=g(e)$.
	Now we are done by Proposition~\ref{lemsameorb}.
	
	Assume now that $H$ and $G$ have the same orbits on directed edges of $\tree$. Let $T \subset \tree$ be an arbitrarily big finite complete subtree of $\tree$. Let $\varphi \in \mathcal{F}(G)$.
	We want to show that there exists a $\psi \in \mathcal{F}(H)$ such that $\varphi \in \psi \cdot \operatorname{Fix}_{G}(T)$.
	Let $T_1,T_2$ be finite complete subtrees of $\tree$ such that $\varphi \colon \tree \setminus T_1 \to \tree \setminus T_2$ is a representative as $G$-honest almost automorphism.
	Assume without loss of generality that $T_1$ contains $T$.
	Let $v$ be a leaf of $T_1$ and let $e$ be the unique edge of $\tree$ not contained in $T_1$ such that $t(e)=v$. Let $g_v \in G$ be such that $g_v|_{\tree_v}=\varphi|_{\tree_v}$.
	Since $H$ and $G$ have the same orbits, there exists an $h_v \in H$ such that $h_v(e)=g_v(e)$. For this $h_v$ holds $h_v(\tree_v)=g_v(\tree_v)$.
	Let $\psi \colon \tree \setminus T_1 \to \tree \setminus T_2$ be the $H$-honest almost automorphism such that for every leaf $v$ of $T_1$ holds $\psi|_{\tree_v}=h_v|_{\tree_v}$.
	By construction $\psi^{-1} \varphi \in \operatorname{Fix}_{G}(T_1) \leq \operatorname{Fix}_{G}(T)$, so we are done.
\end{proof}

\begin{example}
	If $\tree$ is regular of degree $d+1$ and $F \leq F' \leq \Sym(D)$ then $\tfg(U(F))$ is dense in $\tfg(U(F'))$ if and only if $F$ and $F'$ have the same orbits. See Section \ref{burgermozesgroups} to recall definitions. In addition, Proposition 3.5 in \cite{lb16} proves that $\tfg(U(F))$ is closed in $\mathcal{N}$ if and only if $F$ is a Young subgroup.
\end{example}

\subsection{Higman-Thompson groups} \label{chht}

We follow the approach of Caprace and De Medts described in \cite{cm11} Section 6.3.
For a general reference with proofs, see Higman's lecture notes \cite{h74}.
Another good introduction is \cite{b87}, Section 4.

\begin{definition}
	A \emph{plane order} on a rooted tree is a collection of total orders $\{<_v \mid v \in \mathcal V\}$ such that
	for each $v \in \mathcal{V}$ the element $<_v$ is a total order on the children of $v$.
\end{definition}
%
%
%

\begin{remark}
	This is called a plane order because it indicates an embedding of the tree into $\mathbb{R}^2$ with the following properties.
	The root is at the origin and the children of each vertex are below its parent, arranged from left to right according to the plane order.
\end{remark}

\begin{definition} \label{locordpresdef}
	An almost automorphism $\varphi \in \AAut(\tree)$ is called \emph{locally order-preserving} if there exist
	finite complete subtrees $T_1,T_2 \subset \tree$ 
	and a representative as honest almost automorphism
	$\varphi \colon \tree \setminus T_1 \to \tree \setminus T_2$ satisfying the following.
	For every vertex $v$ of $\tree \setminus T_1$ the restriction of $\varphi$ on the children of $v$ is order-preserving.
\end{definition}

It is not hard to see that the set of locally order-preserving almost automorphisms is a subgroup which does not depend on the root, but on the plane order.

\begin{definition} \label{htdef}
	Let $\tree$ be such that the root has $k$ children and all other vertices have $d$ children.
	The subgroup of $\AAut(\tree)$ of locally order-preserving elements is called the \emph{Higman-Thompson group} $V_{d,k}$.
\end{definition}

The conjugacy class of the Higman-Thompson group inside $\AAut(\tree)$ does not depend on the plane order.
In Section \ref{chvf} we will specify a plane order or a regular tree giving us a copy of the Higman-Thompson group which will turn out to be useful for our purposes.

\begin{remark} \label{htdetermindrem}
	Let $T_1,T_2 \subset \tree$ be finite complete subtrees with the same number of leaves.
	Let $\kappa\colon \mathcal{L}T_1 \to \mathcal{L}T_2$ be a bijection.
	Then there exists a unique honest almost automorphism $\varphi \colon \tree \setminus T_1 \to \tree \setminus T_2$ 
	extending $\kappa$ in an order-preserving way.
	We call its equivalence class the almost automorphism \emph{induced by} $\kappa$.
	Let $T_1'$ be a finite complete subtree of $\tree$ containing $T_1$
	and let $T_2'$ be the finite complete subtree of $\tree$ spanned by $\varphi(\mathcal{L}T_1')$.
	Then $\varphi$ is also induced by the bijection
	$\varphi|_{\mathcal{L}T_1'} \colon  \mathcal{L}T_1' \to \mathcal{L}T_2'$.
\end{remark}

\begin{notation}
	For any group $\Lambda$, we denote by $D(\Lambda)$ its commutator subgroup, also called the derived subgroup.
\end{notation}

\paragraph{Abelianization of $V_{d,k}$.} Higman proved that $V_{d,k}$ is finitely presented and that it is simple if $d$ is even and has a simple subgroup of index $2$, its commutator subgroup $D(V_{d,k})$, if $d$ is odd. 
We will now describe the quotient map $V_{d,k} \to V_{d,k}/D(V_{d,k})$.
We will not directly need this quotient map for the present work, but we will generalize the concept in Section~\ref{chspl} and therefore present the idea here. First we need to extend the plane order on $\tree$ to a total order on $\mathcal{V}$.

\begin{definition}\label{deflex}
	Let $\{<_v \mid v \in \mathcal{V}\}$ be a plane order on $\tree$.
	The \emph{lexographical order} on $\tree$ is the total order $<$ on $\mathcal{V}$ defined as follows.
	The choice of the root $v_0$ induces a partial order $\prec$ on $\mathcal{V}$, namely $v \prec w$ if and only if
	$\tree_v \subset \tree_w$. Note that $\prec$ is a join-semilattice, i.e. every finite set has a supremum.
	\begin{itemize}
		\item If $v_1,v_2 \in \mathcal{V}$ are such that $v_1 \prec v_2$, then $v_1 < v_2$.
		\item Otherwise, let $v$ be the supremum of $v_1$ and $v_2$ with respect to $\prec$. 
		For $i=1,2$ let $v_i'$ be the child of $v$ satisfying $v_i \prec v_i'$.
		Then $v_1 < v_2$ if and only if $v_1' <_v v_2'$.
	\end{itemize}
\end{definition}

Let $T_1,T_2,\varphi,\kappa$ as in Defintion \ref{locordpresdef}.
Let $\iota \colon \mathcal{L}T_2 \to \mathcal{L}T_1$ be the unique order-preserving bijection with respect to the lexographical order on $\mathcal{V}$.
Then $\iota \circ \kappa$ is a permutation of the elements in $\mathcal{L}T_1$ and we can consider its sign $\sgn(\iota \circ \kappa) \in \{1,-1\}$.
We want to know when the map $\varphi \mapsto \sgn(\iota \circ \kappa)$ descends to a well-defined homomorphism $V_{d,k} \to \{1,-1\}$.

We replace $T_1$ (and thus also $T_2$) by a simple expansion $T_1'$ (respectively $T_2'$).
We denote by $\kappa'$ the bijection $\mathcal{L}T_1' \to \mathcal{L}T_2'$ induced by $\varphi$
and by $\iota' \colon \mathcal{L}T'_2 \to \mathcal{L}T'_1$ the unique order-preserving bijection with respect to the lexographical order on $\mathcal{V}$.
If $d$ is odd, then $\sgn(\iota' \circ \kappa')=\sgn(\iota \circ \kappa)$.
Therefore the map $\varphi \mapsto \sgn(\iota \circ \kappa)$ descends to a homomorphism $V_{d,k} \to \{1,-1\}$.
The kernel of this homomorphism is the subgroup $D(V_{d,k})$, which is simple as proven by Higman.
If $d$ is even, however, it is not difficult to see that $\sgn(\iota' \circ \kappa')=1$.
Therefore the map $\varphi \mapsto \sgn(\iota \circ \kappa)$ does not descend to a well-defined homomorphism $V_{d,k} \to \{1,-1\}$, but $V_{d,k}$ is simple, see~\cite{h74}.


\section{From almost automorphisms to shifts} \label{chgenht}

In this section we will define the topological full group of an \'etale groupoid as introduced by Matui.
We will see how certain groups of tree almost automorphisms are isomorphic to topological full groups of groupoids associated to one-sided shifts.

\subsection{The topological full group of a groupoid}

We refer to the preliminaries in Matui's article \cite{m12} for a more detailed introduction to this topic.

\paragraph{Topological groupoids.}
A \emph{groupoid} is a category such that every morphism is an isomorphism. 
For our purposes we assume in addition that it is a small category, i.e. the class of objects as well as the class of morphisms are sets.
A \emph{topological groupoid} is a groupoid $\mathcal{G}$ such that the set of objects and the set of morphisms are topological spaces and all structure maps (composition, inverse, identity, source and range) are continuous, and such that source and range are open maps.
We denote by $\mathcal{G}^{(0)}$ the space of objects and by $\mathcal{G}^{(1)}$ the space of morphisms of $\mathcal{G}$.
Denote by $$s,r \colon \mathcal G^{(1)} \to \mathcal{G}^{(0)}$$ the source and range maps. A topological groupoid is called \emph{\'etale} if $s$ and $r$
are local homeomorphisms. We will assume in addition that $\mathcal{G}^{(0)}$ is a Cantor space and $G^{(1)}$ is Hausdorff.

\begin{definition}
	Let $Y \subset \mathcal{G}^{(0)}$ be a clopen subset. The \index{reduction (groupoid)} \emph{reduction of $\gpoid$ to $Y$} is the subgroupoid of $\gpoid$ with object space $Y$ and morphism space $\{g \in \mathcal{G}^{(1)} \mid s(g) \in Y, r(g) \in Y\}$, both endowed with the subspace topology.
	We denote it by $\gpoid|_Y$.
\end{definition}

\begin{definition}
	Let $\mathcal{G}$ be an \'etale groupoid.
	A \emph{bisection} of $\mathcal{G}$ is a clopen subset $U \subset \mathcal{G}^{(1)}$
	such that $s|_U \colon U \to \mathcal{G}^{(0)}$ and $r|_U \colon U \to \mathcal{G}^{(0)}$ are homeomorphisms.
\end{definition}

\begin{definition}
	The \emph{topological full group} of an \'etale groupoid $\mathcal G$ is 
	$$\tfg(\mathcal{G}) := \big\{ r \circ (s|_U)^{-1} \in \operatorname{Homeo}(\mathcal{G}^{(0)}) \bigm| U \subset \mathcal G^{(1)} \text{ bisection of } \mathcal{G} \big\}.$$
	We leave to the reader to check that it is indeed a subgroup of $\operatorname{Homeo}(\mathcal{G}^{(0)})$.
\end{definition}

Recall that in Definition \ref{deftopfull} we already had the notion of a topological full group, namely of a group acting on a topological space. Given an action of a discrete group on the Cantor set, one can associate to it the so-called \emph{action groupoid}, and it turns out that the topological full group of the action groupoid coincides with the topological full group of the group action.
%

\subsection{One-sided shifts of finite type} \label{chshifts}

We refer to \cite{m15}, Section 6, for a more detailed treatment 
of shifts of finite type and the topological full group associated to them.

\begin{definition}
	Let $\mathfrak{g}=(\mathfrak{v},\mathfrak{e})$ be an oriented graph and $E' \subset E$ its orientation.
	The \emph{adjacency matrix} of $\mathfrak{g}$ is the matrix $M_\mathfrak{g} \in \mathbb{Z}^{\mathfrak{v} \times \mathfrak{v}}$ such that
	$M_\mathfrak{g}(v,w)=|\{e \in \mathfrak{e}' \mid i(e)=v,t(e)=w\}|$ for all $v,w \in \mathfrak{v}$.
\end{definition}

\paragraph{Assumptions on oriented graphs.}
Let $\mathfrak{g}=(\mathfrak{v},\mathfrak{e})$ be an oriented graph and $M_\mathfrak{g}$ its adjacency matrix.
We will always require two conditions on $g$ respectively $M_\mathfrak{g}$.
The first condition is that it must be irreducible, 
i.e. for all $v,w \in \mathfrak{v}$ there exists an $n$ such that $M_\mathfrak{g}^n(v,w)\neq 0$.
This is equivalent to saying that there exists a path of length $n$ from $v$ to $w$.
The second condition is that $M_\mathfrak{g}$ must not be a permutation matrix, 
which is equivalent to saying that $\mathfrak g$ is not a disjoint union of oriented cycles.

\paragraph{One-sided irreducible shifts of finite type.}
Let $$X_\mathfrak{g} = \big\{(e_k) \in (\mathfrak{e}')^{\mathbb N} \bigm| i(e_{k+1}) = t(e_k) \big\}$$ be the set of infinite oriented paths in $\mathfrak g$.
Note that $X_\mathfrak{g} \subset (\mathfrak{e}')^{\mathbb N}$ is closed. 
Moreover, that $M_\mathfrak{g}$ is irreducible and not a permutation matrix ensures that $X_\mathfrak{g}$ is a Cantor space.
Define the map $\sigma \colon X_\mathfrak{g} \to X_\mathfrak{g}$ by $\sigma(e)_k = e_{k+1}$. It is a local homeomorphism. 
The pair $(X_\mathfrak{g},\sigma)$ is called the \emph{one-sided irreducible shift of finite type associated to $\mathfrak g$}.

\paragraph{Associated groupoid.}
We associate to $(X_\mathfrak{g},\sigma)$ the following groupoid $\mathcal{G}_\mathfrak{g}$. 
The space of objects and morphisms are
\begin{align*}
\mathcal{G}_\mathfrak{g}^{(0)} &= X_\mathfrak{g} \\
\mathcal{G}_\mathfrak{g}^{(1)} &= \big\{(x,n-m,y) \in X_\mathfrak{g} \times \mathbb Z \times X_\mathfrak{g} \bigm| \sigma^n(x)=\sigma^m(y)\big\} \subset X_\mathfrak{g} \times \mathbb Z \times X_\mathfrak{g}.
\end{align*}
We endow $\mathcal{G}_\mathfrak{g}^{(1)}$ with the topology that is generated by all sets of the form $\{(x,n-m,y) \mid x \in U,\, y \in V,\, \sigma^n(x)=\sigma^m(y)\}$ with $U,V \subset X_\mathfrak{g}$ clopen.
The source and range maps are the projection on the last respecively first factor.
Two elements $(x,n-m,y)$ and $(y',m-l,z)$ are composable if and only if $y=y'$ and the product is $(x,n-m,y) \cdot (y,m-k,z) = (x,n-k,z)$. 
The unit space consists of all elements of the form $(x,0,x)$ and is homeomorphic to $X_\mathfrak{g}$ in an obvious way.
The inverse is given by $(x,n-m,y)^{-1} = (y,m-n,x)$. 

\begin{theorem}[\cite{m15}, Section 6] \label{shiftthm}
	Let $\mathfrak{g}$ be a finite oriented graph such that the associated
	adjacency matrix $M_{\mathfrak{g}}$ is irreducible and not a permutation matrix.
	Then, the topological full group $\tfg(\mathcal G_\mathfrak{g})$ is finitely presented 
	(more precisely, it is of type $F_\infty$).
	Moreover, every non-trivial subgroup of $\tfg(\mathcal G_\mathfrak{g})$ normalized by $D(\tfg(\mathcal G_\mathfrak{g}))$ contains $D(\tfg(\mathcal G_\mathfrak{g}))$.
	In particular $D(\tfg(\mathcal G_\mathfrak{g}))$ is simple.
	Its abelianization is isomorphic to
	\[ \tfg(\mathcal G_\mathfrak{g})/D(\tfg(\mathcal G_\mathfrak{g}))  \cong  \left( \operatorname{Coker}(id - M_{\mathfrak{g}}^t) \otimes_{\mathbb{Z}} \mathbb{Z}/2\mathbb{Z} \right) \oplus \operatorname{Ker}(id - M_{\mathfrak{g}}^t).\]
\end{theorem}

\subsection{A connection between shifts and almost automorphisms}

Recall from Subsection \ref{chht} that the group of locally order-preserving almost automorphisms of a tree depends on the choice of a plane order on the tree.
Also recall that two finite complete subtrees $T_1,T_2$ of $\tree$ and a bijection $\mathcal{L}T_1 \to \mathcal{L}T_2$ incude  a locally order-preserving almost automorphism.

\begin{definition}\label{defvell}
	Let $\tree=(\vertices,\edges)$ be a rooted tree with root $v_0 \in \vertices$.
	A \index{labelling} \emph{labelling} of $\tree$ is a finite set $\mathcal{D}$ together with a map $\ell \colon \vertices \setminus\{v_0\} \to \mathcal{D}$.
	Assume now that $\tree$ is endowed with a plane order.
	Let $T_1,T_2$ be complete finite subtrees of $\tree$ and let $\varphi \colon \tree \setminus T_1 \to \tree \setminus T_2$ be an honest almost automorphism of $\tree$. It is called \index{label-preserving} \emph{label-preserving} if and only if
	\begin{itemize}
		\item it is induced by the bijection $\varphi|_{\mathcal{L}T_1} \colon \mathcal{L}T_1 \to \mathcal{L}T_2$ and
		\item for every $v \in \mathcal{L}T_1$ holds $\ell(v)=\ell(\varphi(v))$.
	\end{itemize}
	We denote the set of all almost automorphisms of $\tree$ admitting a label-preserving representative by $V_\ell$.
\end{definition}

It is in general not true that $V_\ell$ is a subgroup of $\AAut(\tree)$,
since the property of being label-preserving does not survive passing to a simple expansion. It is, however, stable under inversion.

\begin{definition}
	With the notation as in Definition \ref{defvell}, say that $\ell$ is  \index{compatible (with plane order)} \emph{compatible with the plane order} if for all vertices $v,w \in \vertices \setminus \{v_0\}$ with $\ell(v)=\ell(w)$ the following holds.
	They have the same number of children, we denote them by $v_1 < \dots < v_k$ and $w_1 < \dots < w_k$ respectively, and $\ell(v_i)=\ell(w_i)$ for every $i=1,\dots,k$.
\end{definition}

\begin{lemma} \label{lemvellgroup}
	Let $\tree=(\vertices,\edges)$ be a rooted tree endowed with a plane order.
	Let $\ell \colon \vertices \setminus \{v_0\} \to \mathcal{D}$ be a labelling of $\tree$ compatible with the plane order.	
	Then $V_\ell$ is a subgroup of $\AAut(\tree)$.
\end{lemma}

\begin{proof}
	Clearly $V_\ell$ contains the identity element
	and is closed under inverting elements.
	Let $\varphi \colon \tree \setminus T_1 \to \tree \setminus T_2$ and
	$\psi \colon \tree \setminus T_3 \to \tree \setminus T_4$ be elements of $V_\ell$ with label-preserving representatives.
	Let $T$ be a finite complete subtree of $\tree$ containing $T_1$ and $T_4$, and let $T'$ be such that $\psi \colon \tree \setminus T' \to \tree \setminus T$ is another representative.
	Then $\varphi \circ \psi$ is induced by $(\varphi \circ \psi)|_{\mathcal{L}T'}$. Moreover, an easy induction shows that for all vertices $v$ of $\tree \setminus T'$ holds $\ell(v)=\ell(\varphi(\psi(v)))$. The result follows.
\end{proof}

We already know one rather trivial example. Namely, if $\tree$ is such that the root has $k$ children and all other vertices have $d$ children, and if $\ell$ is any constant map, then $V_\ell$ is the Higman-Thompson group $V_{d,k}$.

\begin{theorem}\label{thmalmauttoshift}
	Let $\tree=(\vertices,\edges)$ be a rooted tree with root $v_0 \in \vertices$.
	Assume $\tree$ is endowed with a plane order.
	Let $\mathcal{D}$ be a finite set.
	Let $\ell \colon \vertices \setminus \{v_0\} \to \mathcal{D}$ be a labelling compatible with the plane order.
	Assume in addition that $\tree$ is for every $\delta \in \mathcal{D}$ spanned by the set $\ell^{-1}(\delta)$.
	
	Then, the set $V_\ell$ is a subgroup of $\AAut(\tree)$.
	Moreover, there exists a diconnected and non-circular finite oriented graph $\frakg$ and a clopen subset $Y \subset \gpoid_\frakg^{(0)}$ such that $V_\ell \cong \tfg(\gpoid_\frakg|_Y)$.
\end{theorem}

\begin{remark} \label{remgraphconstr}
	Before going to the proof, we want to construct a suitable graph $\mathfrak{g}=(\mathfrak{v},\mathfrak{e})$ and give an informal motivation why it looks like that. Concrete examples will be given in Theorem \ref{thmnfshift}.
	
	The group $V_\ell$ acts by homeomorphisms on $\partial\tree$, which is a set of infinite paths.
	Similarly the group $\tfg(\mathcal G_{\mathfrak{g}}|_Y)$ acts by homeomorphisms on $Y \subset X_{\mathfrak{g}}$, which is also a set of infinite paths.
	We want to identify $\partial\tree$ and $Y$ in a way that identifies $V_\ell$ and $\tfg(\mathcal G_{\mathfrak{g}}|_Y)$.
	
	Let $v \neq v_0$ be a vertex of $\tree$.
	Consider the finite rooted geodesic with endpoint $v$.
	This path has as many possibilities to continue without backtracking as $v$ has children. The number of children of $v$ only depends on the label of $v$, and each of this children has itself a label, which may or may not be the same as $\ell(v)$.
	
	Implementing this simple observation into the directed graph $\frakg=(\mathfrak{v},\mathfrak{e})$ to be constructed, we declare that for every $\delta \in \mathcal{D}$ the graph $\mathfrak{g}$ has one vertex which we can identify with $\delta$, i.e. $\mathcal{D}\subset \mathfrak{v}$.
	Imitating the children of the vertex $v$ above, we want that the vertex $\ell(v)$ of $\mathfrak{g}$ has as many outgoing edges as $v$ has children, namely one with target $\delta$ for every child of $v$ with label $\delta$. Note that the children of $v$ that have the same label as $v$ will yield loops.
	
	This graph now almost does what we want, but there is still an issue, namely with the root of $\tree$.
	Among its children, some labels might appear several times, and others not at all.
	To deal with the first problem, let for every $\delta \in \mathcal{D}$ the number $l_\delta \in \mathbb{N}$ denote the number of children of $v_0$ with label $\delta$.
	If $l_\delta\geq 2$ we distribute $l_\delta-1$ additional vertices $\nu^{\delta}_{1},\dots,\nu^{\delta}_{l_\delta-1}$ arbitrarily on the incoming edges of $\delta$. This completes the (not completely determined) construction of $\frakg$.
	
	Summarizing, the graph we constructed above can be described and characterized as follows.
	\begin{enumerate}
		\item The set of vertices of $\frakg$ is $\mathfrak{v}=\mathcal{D} \cup \bigcup_{i=1}^{l_\delta-1} \{\nu^{\delta}_{i}\}$.
		\item Call a path starting and ending at an element of $\mathcal{D}$, but not passing through any element of $\mathcal{D}$, an \emph{interrupted edge}.
		\item For every $\delta, \delta' \in \mathcal{D}$ there are exactly as many interrupted edges starting at $\delta$ and ending in $\delta'$ as one, and hence every, vertex $v$ of $\tree$ of label $\delta$ has children with label $\delta'$.
		\item Two interrupted edges intersect at most in their start- and/or endpoint.
		\item Every $\nu^\delta_i$ lies on exactly one interrupted edge. This interrupted edge ends in $\delta$. In particular $\nu^\delta_i$ has incoming degree and outgoing degree equal to one.
		\item The set of edges of $\frakg$ we denote by $\mathfrak{e}$.
	\end{enumerate}
	
	%
	The second problem we solve by choosing the clopen set $Y \subset X_\frakg$ appropriately. Denote $\mathcal{D}_0:=\{\delta \in \mathcal{D} \mid l_\delta =0\}$.
	Then we define $Y$ as the set of all infinite oriented paths in $\frakg$ not starting at a $\delta \in \mathcal{D}_0$, i.e.
	$$Y := \{(e_0,e_1,\dots) \in \mathfrak{e}^\mathbb{N} \mid \forall i \geq 1 \colon o(e_i)=t(e_{i-1}), \, o(e_0) \notin \mathcal{D}_0\}.$$
	%
	%
	%
\end{remark}

\begin{proof}[Proof of Theorem \ref{thmalmauttoshift}]
	Let $\frakg=(\mathfrak{v},\mathfrak{e})$ and $Y$ be in Remark \ref{remgraphconstr}. We will also use other notations from there.
	Let $\mathcal{P}(\mathfrak{g})$ be the set of finite oriented paths of positive length in $\mathfrak{g}$ with starting point not in $\mathcal{D}_0$ and with endpoint in $\mathcal{D}$, i.e.
	$$\mathcal{P}(\mathfrak{g}) := \bigcup_{k \geq 0} \{(e_0,\dots,e_k) \in \mathfrak{e}^{k+1} \mid 
	o(e_i)=t(e_{i-1}), \, o(e_0) \notin \mathcal{D}_0, \, t(e_k) \in \mathcal{D}\}.$$
	The outline of the proof is the following.
	We first construct a tree $\tilde \tree$ with vertex set $\mathcal{P}(\mathfrak{g}) \cup \{\emptyset\}$,
	where we consider $\emptyset$ the root.
	The set of infinite paths $Y$ can be identified with the set of infinite paths $\partial \tilde \tree$ in an obvious way.
	We endow $\tilde \tree$ with a specific plane order.
	Then there exists a unique tree isomorphism $\omega \colon \tilde \tree \to \tree$ with $\omega(\emptyset)=v_0$ and preserving the order on the children of every vertex.
	The plane order on $\tilde \tree$ will have the property that for all $\gamma \in \mathcal{P}(\mathfrak{g})$ 
	holds $\ell(\omega(\gamma)) = t(\gamma)$, where $t(\gamma)$ denotes the terminal vertex of the oriented path $\gamma$.
	We will see that the elements of $\tfg(\mathcal{G}_{\mathfrak{g}}|_Y)$ can be written as precisely those almost automorphisms
	of $\tilde \tree$ which under $\omega$ correspond to elements of $V_\ell$.
	
	Let $\delta \in \mathcal{D}$.
	Let $\Omega^{\delta}$ be the set of interrupted edges starting at $\delta$.
	
	Now we are ready to construct $\tilde \tree$. As mentioned we declare $\emptyset$ to be the root of $\tilde \tree$.
	We define the set of children of $\emptyset$ to consist of the minimal paths in $\mathcal{P}(\frakg)$, i.e.
	$$\{(e_0,\dots,e_k) \in \mathcal{P}(\frakg) \mid e_i \in \mathfrak{e}, \,
	\forall i < k \colon   t(e_i) \notin \mathcal{D} \}.$$
	Now iteratively for every $\gamma \in \mathcal{P}(\mathfrak{g})$ with $t(\gamma)=\delta$ 
	the set of children of $\gamma$ is $$\big\{(\gamma,e) \bigm| e \in \Omega^{\delta} \big\}.$$
	Informally speaking every child of the path $\gamma$ is $\gamma$ continued one (possibly interrupted) step further. This completes the construction of $\tilde \tree$.
	
	Next we define a plane order on $\tilde \tree$, i.e. an order on the children of every vertex.
	Note that the root $\emptyset$ has for every $\delta \in \mathcal{D}$ exactly
	$l_\delta$ children with endpoint $\delta$, which is as many children as $v_0 \in \tree$ has with label $\delta$.
	Therefore there exists an order on the children of $\emptyset$ 
	such that the order-preserving bijection $\eta$ from the children of $\emptyset$ to the children of $v_0$ satisfies
	$\ell(\eta(\gamma)) = t(\gamma)$
	for every child $\gamma$ of $\emptyset$.
	Endow the set of children of $\emptyset$ with such an order.
	
	%
	To define an order on the children of the other vertices of $\tilde \tree$, we need a little preparation.
	For every $\delta \in \mathcal{D}$
	fix an arbitrary vertex $v^{\delta} \in \mathcal{V}$ with $\ell(v^{\delta}) =\delta$. It will now serve as reference vertex.
	Further choose a bijection $\zeta^{\delta}$ from $\Omega^{\delta}$ to the children of $v^{\delta}$
	with the property that $\ell(\zeta^{\delta}(\gamma)) = t(\gamma)$ for every $\gamma \in \Omega^{\delta}$.
	Recall that there is a total order on the children of $v^\delta$, since $\tree$ has a plane order.
	Going back to $\tilde \tree$, for every vertex $\gamma$ of $\tilde \tree$ which is not the root its set of children has the form $\{(\gamma,e) \mid e \in \Omega^{t(\gamma)}\}$.
	Let now $e,e' \in \Omega^{t(\gamma)}$. We say that $(\gamma,e) < (\gamma,e')$ if and only if $\zeta^{t(\gamma)}(e) < \zeta^{t(\gamma)}(e')$.
	
	Now we will show that elements of $\tfg(\mathcal{G}_{\mathfrak{g}}|_Y)$ are locally order-preserving almost automorphisms of $\tilde \tree$.
	There is an easy identification between the set of infinite paths $Y$ and the boundary $\partial \tilde \tree$. Namely, note that every element in $Y$ can be uniquely written as $(e_0,e_1,e_2,\dots)$, where $e_0$ is a minimal path in $\mathcal{P}(\frakg)$ and $e_i$ is an interrupted edge in $\Omega^{t(e_{i-1})}$ for every $i \geq 1$. Using this representation, $Y \to \partial \tilde \tree, (e_0,e_1,e_2,\dots) \mapsto (e_0,(e_0,e_1),(e_0,e_1,e_2),\dots)$ is a homeomorphism.
	
	Let $U \subset \mathcal{G}_{\mathfrak{g}}|_Y$ be a bisection.
	There exist clopen partitions $\{U_1,\dots,U_n\}$ and $\{U'_1,\dots,U'_n\}$ of $Y$ and 
	positive integers $n_1,\dots,n_n,m_1,\dots,m_n$ satisfying the following.
	The bisection $U$ can be written as $$U = \bigsqcup_{k=1}^n U_{k} \times \{n_k - m_k \} \times U'_k$$
	and $r \circ (s|_{U})^{-1}$ restricts to a homeomorphism $U_k \to U'_k$ for every $k=1,\dots,n$.
	By making the $U_i$ and $U'_i$ smaller if necessary, we can assume that for every $k=1,\dots,n$ there exist finite paths $\gamma_k=(e_{k0},\dots,e_{kn_k})$ and
	$\gamma_k'=(e'_{k0},\dots,e'_{km_k})$
	such that
	\begin{align*}
	U_k &= \{(e_i)_{i \in \mathbb{N}} \in Y \mid \forall 0 \leq i \leq n_k \colon e_i = e_{ki} \} \\
	U'_k &= \{(e'_i)_{i \in \mathbb{N}} \in Y \mid \forall 0 \leq i \leq m_k \colon e'_i = e'_{ki} \}
	\end{align*}
	and
	$r \circ (s|_{U})^{-1}(\gamma_k,e_{n_k+1},e_{n_k+2},\dots)
	= (\gamma_k',e_{n_k+1},e_{n_k+2},\dots).$
	%
	Note that since the vertices of $\mathfrak{g}$ of the form $\nu_j^{\delta}$ have
	precisely one outgoing edge we can assume that $t(e_{kn_k}),t(e'_{km_k}) \in \mathcal{D}$ for all $k=1,\dots,n$ without changing $U_k$ and $U'_k$.
	Then all $\gamma_k$ and $\gamma'_k$
	are vertices of $\tilde \tree$.
	Observe that for $k=0,\dots,n$ holds
	\begin{align*}
	U_k &= \partial \tilde \tree_{\gamma_k} \\
	U'_k &= \partial \tilde \tree_{\gamma'_k}.
	\end{align*}
	Since $\{U_1,\dots,U_n\}$ is a clopen partition of $Y$, there is a finite complete subtree $T \subset \tilde \tree$ with 
	$$\mathcal{L}T = \{\gamma_k \mid k = 1,\dots,n\}.$$
	In the same way there exists a finite complete subtree $T' \subset \tilde \tree$ with 
	$$\mathcal{L}T' = \{\gamma_k' \mid k = 1,\dots,n\}.$$
	The element $r \circ (s|_{U})^{-1}$ is then an almost automorphism $\tilde \tree \setminus T \to \tilde \tree \setminus T'$.
	It is such that for every vertex 
	$\tilde \gamma = (\gamma_k,e_{n_k+1},e_{n_k+2},\dots,e_{i'})$ 
	%
	of $\tilde \tree \setminus T$ holds 
	$$r \circ (s|_{U})^{-1}(\tilde \gamma)=(\gamma'_k,e_{n_k+1},e_{n_k+2},\dots,e_{i'}).$$
	Therefore this almost automorphism is locally order-preserving.
	This shows that every element of $\tfg(\mathcal{G}_{\mathfrak{g}}|_Y)$ is a locally order-preserving almost automorphism of $\tilde \tree$.
	
	Let now $n \geq 0$ and for $k=0,\dots,n$ let $\gamma_k=(e_{k0},\dots,e_{kn_k}) \in \mathcal{P}(\mathfrak{g})$ and
	$\gamma'_k=(e'_{k0},\dots,e'_{km_k}) \in \mathcal{P}(\mathfrak{g})$
	be finite paths such that for
	\begin{align*}
	U_k &= \{(e_i)_{i \in \mathbb{N}} \in Y \mid \forall 0 \leq i \leq n_k \colon e_i = e_{ki} \} \\
	U'_k &= \{(e'_i)_{i \in \mathbb{N}} \in Y \mid \forall 0 \leq i \leq m_k \colon e'_i = e'_{ki} \}
	\end{align*}
	the sets $\{U_1,\dots,U_n\}$ and $\{U'_1,\dots,U'_n\}$ are clopen partitions of $Y$.
	Then it is easy to see that that
	$\bigsqcup_{k=1}^n U_k \times \{n_k - m_k\} \times U'_k$
	is a bisection of $\gpoid_\frakg|_Y$ if and only if $t(\gamma_k)=t(\gamma_k')$
	for every $k$.
	
	Let, on the other hand, $T,T'\subset \tilde \tree$ be finite complete subtrees and let $\varphi \colon \tilde \tree \setminus T \to \tilde \tree \setminus T'$ be an honest almost automorphism satisfying the following. 
	It is locally order-preserving and for all $\gamma \in \mathcal{L}T$ holds $t(\gamma)=t(\varphi(\gamma))$.
	Repeating above argument in reverse shows that
	the sets of leaves $\mathcal{L}T$ and $\mathcal{L}T'$ define two compact and open partitions of $\partial \tilde \tree$.
	These partitions together with the bijection $\varphi|_{\mathcal{L}T} \colon \mathcal{L}T \to \mathcal{L}T'$ can also be interpreted as a bisection of $\mathcal{G}_{\mathfrak{g}}|_Y$. This gives rise to an element of $\tfg(\mathcal{G}_{\mathfrak{g}}|_Y)$ looking like $\varphi$.
	
	Let $\omega \colon \tilde \tree \to \tree$ 
	be the unique tree isomorphism with $\omega(\emptyset)=v_0$ and 
	preserving the order on the children on every vertex. 
	It has the following property. 
	By definition of the plane order on $\tilde \tree$,
	for every $\gamma \in \mathcal{P}(\mathfrak{g})$ holds
	$\ell(\omega(\gamma)) = t(\gamma)$.
	Therefore, the almost automorphisms in
	$\tfg(\mathcal{G}_{\mathfrak{g}}|_Y)$ and the almost automorphisms in $V_\ell$ exactly correspond to each other under $\omega$.
\end{proof}

\subsection{A subgroup of  $\tfg(U(F))$}

Let as before $d \geq 2$ be an integer and $D:=\{0,\dots,d\}$. In this section $\tree$ is a $(d+1)$-regular tree. Let $F \leq \Sym(D)$ be any subgroup. We use notation and definitions from Subsection \ref{burgermozesgroups}.

\begin{notation}
	From now on we will write $\UF := \mathcal{F}(U(F))$.
\end{notation}

In this section we investigate a certain subgroup $V_F \leq \UF$ that 
plays a role analogous to the Higman-Thompson group inside Neretin's group.
Theorem \ref{thmalmauttoshift} will give that it is isomorphic to the topological full group of a shift of finite type.
Consequently, by results of Matui \cite{m15},
it is finitely presented and its commutator subgroup $D(V_F)$ is simple.

\begin{notation}
	We denote the orbits of $F$ by $D^{(0)},\dots,D^{(l)} \subset D$.
	For each $i=0,\dots,l$ we write $d^{(i)} := |D^{(i)}|$.
\end{notation}

\subsubsection{A plane order on $\tree$} \label{chvf}

Let $T_1, T_2$ be finite complete subtrees of $\tree$ and let
$\varphi \colon \tree \setminus T_1 \to \tree \setminus T_2$ be an arbitrary $U(F)$-honest almost automorphism.
Then $\varphi|_{\mathcal{L}T_1} \colon \mathcal{L}T_1 \to \mathcal{L}T_2$ is a bijection
such that for every $v \in \mathcal{L}T_1$ the colours of the parent edges of $v$ and $\varphi(v)$
are in the same orbit of $F$.
This motivates the following definition of a labelling on $\tree$.

\begin{definition} \label{deflabellinguf}
	Let $\ell_{F} \colon \vertices \setminus \{v_0\} \to \{D^{(0)},\dots,D^{(l)}\}$ be defined as follows.
	Let $v \neq v_0$ be a vertex of $\tree$ and $e$ its parent edge. Then $\ell_{F}(v)$ is the $F$-orbit of $\col(e)$.
\end{definition}

%

Now we construct a plane order on $\tree$ such that we can apply the results of the preceding subsection to investigate the set of all label-preserving almost automorphisms $V_{\ell_{F}}$.
%
%

\begin{proposition} \label{propord}
	There exists a plane order on $\mathcal{T}$ with $V_{\ell_{F}} = V_{d,d+1} \cap \UF$.
	
	More precisely, there exists a plane order on $\tree$ such that
	every label-preserving honest almost automorphism is an
	$U(F)$-honest almost automorphism.
\end{proposition}

\begin{proof}
	As we will now see, the inclusion $V_{\ell_{F}} \supset V_{d,d+1} \cap \UF$ holds independently of the plane order.
	Let $\psi \in V_{d,d+1} \cap \UF$.
	Then there exist finite complete subtrees $T_1,T_2 \subset \tree$
	such that $\psi \colon  \tree \setminus T_1 \to \tree \setminus T_2$ 
	is an $U(F)$-almost automorphism which is
	as element of $V_{d,d+1}$ induced by the bijection
	$\psi|_{\mathcal{L}T_1} \colon \mathcal{L}T_1 \to \mathcal{L}T_2$.
	It is enough to show that the bijection $\psi|_{\mathcal{L}T_1}$ enjoys the following property.
	For all $v \in \mathcal{L}T_1$ and for all elements $g \in U(F)$
	such that $\psi|_{\tree_v}=g|_{\tree_v}$
	holds $\ell_F(v)=\ell_F(gv)$.
	But note that this is true because for the parent edge $e$ of $v$ holds $\prm_{g,v}(\col(e))=\col(g(e))$.
	
	For the reverse inclusion we first specify a plane order on $\mathcal{T}$, i.e. a total order $<_v$ on the children of every vertex $v \in \mathcal{V}$, and then prove that it has the desired property.
	Denote for every vertex $v' \in \mathcal{V}\setminus\{v_0\}$ its parent edge by $e_{v'}$.
	
	For children of $v_0$ we say that $v <_{v_0} w$ if and only if $\col(e_v) < \col(e_w)$.
	For all other vertices $v \in \mathcal{V}$ the order $<_v$ will be determined by $\col(e_v)$.
	Choose for each $i=0,\dots,l$ a colour $\chi^{(i)} \in D^{(i)}$, which we will now use as ``reference colour''.
	Let $v$ be any vertex with $\col(e_v)=\chi^{(i)}$ and let $v_1,v_2$ be children of $v$.
	Say $v_1 <_{v} v_2$ if and only if $\col(e_{v_1}) < \col(e_{v_2})$.
	Let now $\chi \in D \setminus \{\chi^{(0)},\dots \chi^{(l)}\}$ and let $i$ be such that $\chi \in D^{(i)}$. We will now compare with the reference colour.
	Choose an element $f_{\chi} \in F$ with $f_\chi(\chi) = \chi^{(i)}$, and set $f_{\chi^{(i)}}=id$.
	Then $f_\chi$ restricts to a bijection $D \setminus \{\chi\} \to D \setminus \{\chi^{(i)}\}$.
	Let $v$ be any vertex with $\col(e_v)=\chi$ and let $v_1,v_2$ be children of $v$.
	Say $v_1 <_{v} v_2$ if and only if $f_{\chi}(\col(e_{v_1})) < f_{\chi}(\col(e_{v_2}))$.
	
	To verify that this plane order has the desired property, we first prove that
	$V_{\ell_{F}} \subset V_{d,d+1} \cap \UF$.
	Let $T_1$ and $T_2$ be two finite complete subtrees of $\tree$ and let $\varphi \colon \tree \setminus T_1 \to \tree \setminus T_2$ be a label-preserving honest almost automorphism.
	Let $v$ be a vertex of $\tree \setminus T_1$ and
	let $g \in \AutT$ be such that $\psi|_{\tree_v}=g|_{\tree_v}$. We have to prove that $\prm_{g,v} \in F$. We do this in three steps.
	
	\emph{Step 1:} The vertex $v$ is a leaf of $T_1$.
	
	By Lemma \ref{lemextuf} there exists an element $h \in U(F)$ with $\col(e_{hv})=\chi^{(i)}$ and $\prm_{h,v}=f_{\col(e_v)}$.
	Then by Lemma \ref{lemmultprm} we have
	$\prm_{g,v} = \prm_{gh^{-1}h,v} = \prm_{gh^{-1},hv} \circ \prm_{h,v}=
	f_{\col(e_{gv})}^{-1}\circ f_{\col(e_v)} \in F$.
	
	\emph{Step 2:} The vertex $v$ is a child of a leaf of $T_1$.
	
	Let $T'_1$ denote the simple expansion of $T_1$ such that $v \in \mathcal{L}T'_1$.
	Similarly denote by $T'_2$ the simple expansion of $T_2$ such that
	$gv \in \mathcal{L}T'_2$.
	From Step~1 follows in particular that for all leaves $w$ of $T_1'$
	the colours $\col(e_w)$ and $\col(e_{gw})$ lie in the same orbit of $F$, and therefore $\ell_F(w)=\ell_F(gw)$.
	Note that $\varphi$ is also induced by the bijection
	$\varphi|_{\mathcal{L}T'_1} \colon \mathcal{L}T'_1 \to \mathcal{L}T'_2$.
	Thus the restriction $\varphi|_{\tree \setminus T_1'} \colon \tree \setminus T_1' \to \tree \setminus T_2'$
	is a label-preserving honest almost automorphism.
	Now we can repeat the argument from Step~1 for
	$\varphi$ replaced by $\varphi|_{\tree \setminus T_1'}$ and
	get that $\prm_{g,v} \in F$.
	
	\emph{Step 3:} The vertex $v$ is a descendant of a vertex of $T_1$.
	
	Recall that every finite complete subtree containing $T_1$ is obtained from $T_1$ by
	a finite sequence of simple expansions. Therefore we iteratively get $\prm_{g,v} \in F$.
	This means that indeed $V_{\ell_{F}} \subset V_{d,d+1} \cap \UF$.
\end{proof}

\begin{notation}
	Henceforth we abbreviate $V_F := V_{\ell_F}$.
\end{notation}

\begin{example}
	Consider the $4$-regular tree from Figure \ref{figemb}.
	Its plane order is implied by how the tree is drawn from left to right, namely,
	for any children $v_1,v_2$ of a vertex $v$ holds $v_1 <_v v_2$ if and only if $v_1$ is drawn to the left of $v_2$.
	
	Let $F = \langle (12) \rangle \leq S_4$.
	Let $T_1=T_2$ be the finite complete subtree whose leaves are the children of $v_0$.
	Then there exists an element in $\UF$ (even in $U(F)$) switching the vertices of label $\{1,2\}$, i.e. the children of $v_0$ whose parent edge have colour $1$ and $2$.
	Note, however, that the element of $V_{d,d+1}$ induced by this permutation of $\mathcal{L}T_1$ is not an element of $\UF$.
	
	Let now $F=\langle (1 \, 2 \, 3) \rangle$.
	Then $V_F = V_{d,d+1} \cap \UF$ holds with the drawn order.
\end{example}

%
%
%
%

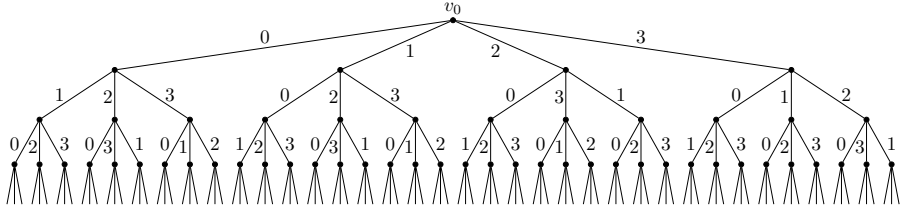
\begin{figure}[H]
	\centering
	\resizebox{120mm}{!}{ 
		\begin{tikzpicture}[font=\normalsize]  [grow'=down]
		
		\tikzstyle{solid node}=[circle,draw,inner sep=1,fill=black];
		\tikzstyle{level 1}=[level distance=0mm,sibling distance=0mm]
		\tikzstyle{level 2}=[level distance=10mm,sibling distance=45mm]
		\tikzstyle{level 3}=[level distance=10mm,sibling distance=15mm]
		\tikzstyle{level 4}=[level distance=9mm,sibling distance=5mm]
		\tikzstyle{level 5}=[level distance=8mm,sibling distance=1.5mm]
		
		\coordinate
		child foreach \A in {1}
		{node(0)[solid node]{} child foreach \B in {red,green,blue,gray}
			{node(1)[solid node]{} child foreach \C in {black,gray,white}
				{node(2)[solid node]{} child foreach \C in {black,gray,white} 
					{node(3)[solid node]{} child foreach \C in {black,gray,white}  }
				}
		}};
		
		\node[above,xshift=0mm]at(0){$v_0$};
		
		\node[above,xshift=-30mm,yshift=4mm]at(1){3};
		\node[above,xshift=-59mm,yshift=1.3mm]at(1){2};
		\node[above,xshift=-76mm,yshift=1.3mm]at(1){1};
		\node[above,xshift=-105mm,yshift=4mm]at(1){0};
		
		\node[above,xshift=-4mm,yshift=2.5mm]at(2){2};
		\node[above left,xshift=-14mm,yshift=2mm]at(2){1};
		\node[above,xshift=-26mm,yshift=2.5mm]at(2){0};
		
		\node[above left,xshift=-104mm,yshift=2mm]at(2){2};
		\node[above,xshift=-94mm,yshift=2.5mm]at(2){3};
		\node[above,xshift=-116mm,yshift=2.5mm]at(2){0};
		
		\node[above,xshift=-49mm,yshift=2.5mm]at(2){1};
		\node[above left,xshift=-59mm,yshift=2mm]at(2){3};
		\node[above,xshift=-71mm,yshift=2.5mm]at(2){0};
		
		\node[above,xshift=-139mm,yshift=2.5mm]at(2){3};
		\node[above left,xshift=-149mm,yshift=2mm]at(2){2};
		\node[above,xshift=-161mm,yshift=2.5mm]at(2){1};
		
		\node[above,xshift=0mm,yshift=1.5mm]at(3){1};
		\node[above left,xshift=-4mm,yshift=1mm]at(3){3};
		\node[above,xshift=-10mm,yshift=1.5mm]at(3){0};
		
		\node[above,xshift=-15mm,yshift=1.5mm]at(3){3};
		\node[above left,xshift=-19mm,yshift=1mm]at(3){2};
		\node[above,xshift=-25mm,yshift=1.5mm]at(3){0};
		
		\node[above,xshift=-30mm,yshift=1.5mm]at(3){3};
		\node[above left,xshift=-34mm,yshift=1mm]at(3){2};
		\node[above,xshift=-40mm,yshift=1.5mm]at(3){1};
		
		\node[above,xshift=-45mm,yshift=1.5mm]at(3){3};
		\node[above left,xshift=-49mm,yshift=1mm]at(3){2};
		\node[above,xshift=-55mm,yshift=1.5mm]at(3){0};
		
		\node[above,xshift=-60mm,yshift=1.5mm]at(3){2};
		\node[above left,xshift=-64mm,yshift=1mm]at(3){1};
		\node[above,xshift=-70mm,yshift=1.5mm]at(3){0};
		
		\node[above,xshift=-75mm,yshift=1.5mm]at(3){3};
		\node[above left,xshift=-79mm,yshift=1mm]at(3){2};
		\node[above,xshift=-85mm,yshift=1.5mm]at(3){1};
		
		\node[above,xshift=-90mm,yshift=1.5mm]at(3){2};
		\node[above left,xshift=-94mm,yshift=1mm]at(3){1};
		\node[above,xshift=-100mm,yshift=1.5mm]at(3){0};
		
		\node[above,xshift=-105mm,yshift=1.5mm]at(3){1};
		\node[above left,xshift=-109mm,yshift=1mm]at(3){3};
		\node[above,xshift=-115mm,yshift=1.5mm]at(3){0};
		
		\node[above,xshift=-120mm,yshift=1.5mm]at(3){3};
		\node[above left,xshift=-124mm,yshift=1mm]at(3){2};
		\node[above,xshift=-130mm,yshift=1.5mm]at(3){1};
		
		\node[above,xshift=-135mm,yshift=1.5mm]at(3){2};
		\node[above left,xshift=-139mm,yshift=1mm]at(3){1};
		\node[above,xshift=-145mm,yshift=1.5mm]at(3){0};
		
		\node[above,xshift=-150mm,yshift=1.5mm]at(3){1};
		\node[above left,xshift=-154mm,yshift=1mm]at(3){3};
		\node[above,xshift=-160mm,yshift=1.5mm]at(3){0};
		
		\node[above,xshift=-165mm,yshift=1.5mm]at(3){3};
		\node[above left,xshift=-169mm,yshift=1mm]at(3){2};
		\node[above,xshift=-175mm,yshift=1.5mm]at(3){0};
		\end{tikzpicture}
	}
	\caption{The left-to-right drawing specifies an order on the tree as in
		Proposition \ref{propord} for $F=\langle (1,2,3) \rangle < S_4$.}
	\label{figemb}
\end{figure}
%

\begin{theorem} \label{thmnfshift}
	There exists a diconnected, non-circular, finite, oriented graph $\frakg$ such that $V_F$ is isomorphic to $\tfg(\gpoid_{\frakg})$.
	
	More precisely, let $\frakg_F$ be the following graph.
	Let $\mathcal{D} := \{D^{(0)},\dots,D^{(l)}\}$.
	The vertex set of $\frakg_F$ is the union $\mathfrak{v}_F = \mathcal{D} \cup \bigcup_{i=0}^l \{\nu_1^{(i)},\dots,\nu_{d^{(i)}-1}^{(i)}\}$, where the $\nu_j^{(i)}$ are arbitrary pairwise different elements.
	The adjacency matrix in $\mathbb{Z}^{\mathfrak{v}_F \times \mathfrak{v}_F}$ is defined to be
	\[
	M_{\mathfrak{g}_F}(v,w) = \begin{cases}
	d^{(i)} & v = D^{(j)}, \, w=D^{(i)}, \, i \neq j \\
	1 & v=D^{(i)}, \, w = \nu_1^{(i)},\dots,\nu_{d^{(i)}-1}^{(i)} \\
	1 & v=\nu_1^{(i)},\dots,\nu_{d^{(i)}-1}^{(i)}, \, w = D^{(i)} \\
	0 & \text{else.}           
	\end{cases}
	\]
	Then $V_F \cong \tfg(\gpoid_{\frakg_F})$.
\end{theorem}

\begin{figure}[H]
	\centering
	\scalebox{1}{
		\begin{tikzpicture}[->,>=stealth',shorten >=1pt,auto,node distance=2.8cm,
		semithick]
		
		\node[state]         (A)                    {$D^{(0)}$};
		\node[state]         (B) [below left  of=A] {$D^{(1)}$};
		\node[state]         (C) [below right of=A] {$D^{(2)}$};
		\node[state]        (D) [left  of=B,inner sep=0pt,minimum size=0pt,node distance=1.5cm]       {$\nu_1^{(1)}$};
		\node[state]        (E) [below of=B,inner sep=0pt,minimum size=0pt,node distance=1.5cm]       {$\nu_2^{(1)}$};
		\node[state]        (F) [right of=C,inner sep=0pt,minimum size=0pt,node distance=1.5cm]       {$\nu_1^{(2)}$};
		
		\path (A) edge [bend right=10]     (B)
		edge [bend right=20]     (B)
		edge [bend right=30]     (B)
		edge [bend left=10]    (C)
		edge [bend left=20]    (C)
		(B) edge [bend left=10]     (C)
		edge [bend left=20]     (C)
		edge [bend right=10]    (A)  
		(C) edge [bend left=10]     (B)
		edge [bend left=20]     (B)
		edge [bend left=30]     (B)
		edge [bend left=10]    (A)
		(B) edge [bend right]  (D)
		(D) edge [bend right] (B)
		(B) edge [bend right]  (E)
		(E) edge [bend right] (B)
		(C) edge [bend left]  (F)
		(F) edge [bend left] (C);
		\end{tikzpicture}}
	\caption{The graph $\mathfrak{g}_F$ for $F < S_6$ with $(d^{(0)},d^{(1)},d^{(2)})=(1,3,2)$}
\end{figure}
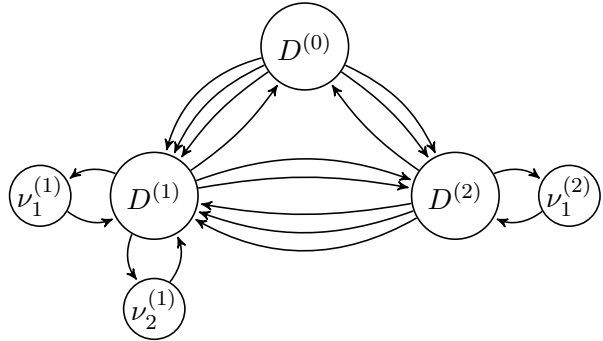

\begin{proof}
	This follows directly from Theorem \ref{thmalmauttoshift} since the graph $\frakg_F$ is exactly as in Remark \ref{remgraphconstr}.
\end{proof}

\begin{remark}
	It might seem surprising at first sight that the dense subgroup $V_F$ only depends on the size of the orbits of $F$.
	However, this is precisely what was already known in the transitive case. If $F$ is transitive, then $\UF$ is a group introduced by Caprace and De Medts in \cite{cm11} which in the literature is usually denoted as $\AAut_D(\tree_{d,2})$, a good introduction to it is Le Boudec's article \cite{lb17}. The group $V_F$ in this case is the Higman-Thompson group $V_{d,2}=V_{d,d+1}$.
\end{remark}

\section{Compact generation and virtual simplicity} \label{chspl}

Let $F \leq \Sym(D)$ be any subgroup.
In this section we prove that $\UF$ is compactly generated and that $D(\UF)$ is open, simple and has finite index in $\UF$.
Compact generation is a direct consequence of the theorem below.
The statement is the analog to saying that Neretin's group contains a dense copy of a Higman-Thompson group.

\begin{theorem} \label{cpctgenthm}
	The finitely generated group $V_F$ is dense in $\UF$.
\end{theorem}
%

\begin{proof}
	Let $T \subset \tree$ be an arbitrarily big finite complete subtree.
	We need to prove that for every element $\varphi \in \UF$ there exists an element $\psi \in V_F$ such that $\psi^{-1} \varphi \in \Fix_{U(F)}(T)$.
	Let $T_1,T_2 \subset \tree$ be two finite complete subtrees and
	let $\varphi \colon \tree \setminus T_1 \to \tree \setminus T_2$ be an $U(F)$-honest almost automorphism.
	By Remark \ref{equivclassrmk} we can assume that $T_1 \supset T$.
	Then $\varphi$ restricts to a bijection $\mathcal{L}T_1 \to \mathcal{L}T_2$ which induces an element 
	$\psi \colon \tree \setminus T_1 \to \tree \setminus T_2$ in $V_{d,d+1}$ 
	such that $\psi|_{\mathcal{L}T_1 } = \varphi|_{\mathcal{L}T_1 }$.
	Proposition \ref{propord} implies that $\psi \in V_F$.
	In addition we observe that
	$\psi^{-1} \varphi \in \Fix_{U(F)}(T)$, which concludes the proof.
\end{proof}

\begin{definition}
	Let a group $\Lambda$ act on a topological space $X$. The action is called \emph{minimal} if every orbit is dense.
\end{definition}

\begin{definition}
	Let a group $\Lambda$ act on the Cantor space $X$. The action is called \emph{purely infinite} if for every nonempty compact open subset $U \subsetneq X$ there exist $g,h \in \Lambda$ such that $g(U) \cup h(U) \subset U$ and $g(U) \cap h(U) = \emptyset$.
\end{definition}

\begin{theorem}[\cite{m15}, Theorem 4.16] \label{thmmatuispl}
	Let a group $\Lambda$ act minimally on the Cantor space such that the action is purely infinite. The commutator subgroup of the topological full group of this action is simple.
\end{theorem}

\begin{remark}
	In the article where this theorem is stated, it is assumed that $G$ is countable and
	the action is essentially free. However, a close inspection of the proof shows that
	these two assumptions are not used. I am grateful to Hiroki Matui for clarifying
	this point with me. See also Theorem 5.1 in~\cite{sg17}.
\end{remark}

\begin{corollary} \label{cor:dnfsple}
	The commutator subgroup $D(\UF)$ is simple.
\end{corollary}

\begin{proof}
	By Theorem \ref{thmmatuispl} it suffices to show that the action $U(F)\curvearrowright \partial \tree$ is purely infinite and minimal.  
	For minimality, we refer to Proposition 51 in \cite{a03}, there it is also shown that the action does no preserve any proper subtree.
	For a proof of the classical fact that this implies that the action is purely infinite, see Lemma 4.25 in \cite{lbmb18}.
\end{proof}

To investigate the abelianization of $V_F$, we need the well-known Smith normal form. For the reader's convenience we recall the statement here, it can be looked up e.g. in \cite{bk00}, Section 3.3.2 or \cite{aw92}, Section 5.3.

\begin{lemma}[Smith normal form] \label{lemsnf}
	Let $R$ be a principal ideal domain and let $M \in R^{m \times n}$ be a matrix. 
	Then, there exist invertible matrices $S \in R^{m \times m}$ and $T \in R^{n \times n}$, 
	an integer $k \leq \min\{m,n\}$ and elements $\epsilon_1,\dots,\epsilon_k \in R$, called \emph{elementary divisors}, 
	such that
	\begin{itemize}
		\item $\epsilon_i$ is the $i$-th diagonal entry of $SMT$,
		\item all other entries of $SMT$ are $0$ and
		\item $\epsilon_i$ divides $\epsilon_{i+1}$ for $k=1,\dots,k-1$.
	\end{itemize}
	The elementary divisors are unique up to multiplication with a unit. 
	They have the property that the product $\epsilon_1 \dots \epsilon_i$ is the greatest common divisor of the determinants of all $i \times i$-submatrices. Furthermore
	\[
	\operatorname{Coker}(M) \cong R^{m-k} \times \prod_{i=1}^k R/R\epsilon_i.
	\]
\end{lemma}

\begin{notation}
	As in the preceeding section we denote the orbits of $F$ by $D^{(0)},\dots,D^{(l)} \subset D$.
	For each $i=0,\dots,l$ we write $d^{(i)} := |D^{(i)}|$.
\end{notation}

\begin{proposition} \label{abprop}
	The commutator subgroup $D(V_F)$ has finite index in $V_F$.
	More precisely, if all $d^{(i)}$ are even, the abelianization of $V_F$ is isomorphic to $(\mathbb{Z}/2\mathbb{Z})^{l+1}$.
	Otherwise it is isomorphic to $(\mathbb{Z}/2\mathbb{Z})^{l}$.
\end{proposition}

\begin{proof}
	By Theorem \ref{shiftthm} the abelianization is isomorphic to 
	\[ V_F/D(V_F)  \cong  \left( \operatorname{Coker}(id - M_{\mathfrak{g}_F}^t) \otimes_{\mathbb{Z}} \mathbb{Z}/2\mathbb{Z} \right) \oplus \operatorname{Ker}(id - M_{\mathfrak{g}_F}^t).\]
	To determine $\operatorname{Coker}(id - M_{\mathfrak{g}_F}^t)$ and $\operatorname{Ker}(id - M_{\mathfrak{g}_F}^t)$ we use the Smith normal form.
	When writing out the matrix $id - M_{\mathfrak{g}_F}^t$ explicitly,
	it is not hard to see that performing elementary row- and column operations on $id-M_{\mathfrak{g}_F}^t$ we get the block diagonal matrix
	\[id - M_{\mathfrak{g}_F}^t \sim \left(\begin{array}{ccccc|cccccc}
	1      & 0      & 0      & \dots  & 0  & 0  & \dots  & \dots  & \dots  & \dots  & 0 \\
	0      & 1      & 0      & \dots  & 0  & 0  & \dots  & \dots  & \dots  & \dots  & 0 \\
	\vdots & \vdots & \ddots & \vdots & \vdots  & \vdots & \vdots & \vdots & \vdots & \vdots  & \vdots \\
	0      & 0      & \dots  & 1  & 0      & 0      & \dots  & \dots  & \dots  & \dots  & 0 \\
	0      & 0      & \dots  & 0  & 1      & 0      & \dots  & \dots  & \dots  & \dots  & 0 \\
	\hline
	0      & 0      & \dots  & \dots  & 0      & 1-d    & 0      & 0  & \dots  & \dots  & 0 \\
	0      & 0      & \dots  & \dots  & 0      & -d^{(1)}    & 2      & 0  & \dots  & \dots  & 0 \\
	0      & 0      & \dots  & \dots  & 0      & -d^{(2)}    & 0      & 2  & \dots  & \dots  & 0 \\
	\vdots & \vdots & \ddots & \vdots & \vdots  & \vdots & \vdots & \vdots & \ddots & \vdots  & \vdots \\
	0      & 0      & \dots  & \dots  & 0      & -d^{(l-1)}   & 0     & 0  & \dots  & 2  & 0 \\
	0      & 0      & \dots  & \dots  & 0      & -d^{(l)}    & 0      & 0  & \dots  & 0  & 2
	\end{array}
	\right)_, \]
	which has a $(d-l)\times(d-l)$ identity matrix in the upper left corner.
	The determinant of this matrix is $2^l \cdot (1-d)$, therefore $\operatorname{Ker}(id - M^t_{\mathfrak{g}_F}) = \{0\}$.
	This already implies that the number of elementary divisors is $d+1$, so
	$\operatorname{Coker}(id - M_{\mathfrak{g}_F}^t)$
	is finite and therefore $D(V_F)$ has finite index in $V_F$.
	
	We now determine the abelianization $V_F/D(V_F)$. Let $\epsilon_1,\dots,\epsilon_{d+1}$ be the elementary divisors of $id - M_{\mathfrak{g}_F}^t$.
	Lemma \ref{lemsnf} says $$\operatorname{Coker}(id - M_{\mathfrak{g}_F}^t) \cong \prod_{i=1}^{d+1} \mathbb{Z}/\epsilon_i \mathbb{Z}.$$
	The first $d-l$ elementary divisors are $\epsilon_1 = \dots = \epsilon_{d-l}=1$.
	The $(d-l+1)$-th to $(d+1)$-th elementary divisors are the elementary divisors of the second of two blocks in the above block diagonal matrix.
	They are given by the greatest common divisors of determinants of submatrices.
	Note that all possibly odd matrix entries are in the same column.
	From that we see that if one of the $d^{(i)}$ is odd, then the $(d-l+1)$-th elementary divisor is odd and the further ones are even, otherwise all are even.
	Now since $\mathbb{Z}/n\mathbb{Z} \otimes_\mathbb{Z} \mathbb{Z}/m\mathbb{Z} \cong \mathbb{Z}/\operatorname{gcd}(n,m)\mathbb{Z}$ and since tensor product is distribuitive with direct sums, we get that 
	\[
	V_F/D(V_F) \cong \prod_{i=1}^{d+1} \mathbb{Z}/\epsilon_i \mathbb{Z} \otimes_\mathbb{Z} \mathbb{Z}/2 \mathbb{Z}
	\cong \begin{cases}
	(\mathbb{Z}/2 \mathbb{Z})^{l+1} & \text{if all } d^{(i)} \text{ are even} \\
	(\mathbb{Z}/2 \mathbb{Z})^{l} & \text{otherwise}.
	\end{cases}
	\]
\end{proof}

\begin{remark}
	This shows in particular that if $F$ is not transitive, then $V_F$ is not isomorphic to any Higman-Thompson group. By Theorem 3.10 in \cite{m15}, this implies that also the commutator subgroup cannot be isomorphic to the commutator subgroup of any Higman-Thompson group.
\end{remark}

\begin{theorem} \label{thm:dnffi}
	The commutator subgroup $D(\mathcal{N}_F)$ of $\mathcal{N}_F$ is open and has finite index.
	More precisely, the homomorphism
	$$V_F/D(V_F) \to \UF/D(\UF), \quad \varphi D(V_F) \mapsto \varphi D(\UF)$$
	is surjective.
\end{theorem}

\begin{proof}
	We first show that $D(\mathcal{N}_F)$ is open. If $U(F)$ is discrete, so is $\UF$ and there is nothing to show.
	Recall that $U(F)^+$ is the subgroup of $U(F)$ generated by all the edge fixators in $U(F)$.
	It is trivial if and only if the action of $F$ on $D$ is free, so if and only if $U(F)$ is discrete, see Remark \ref{rmkufplustriv}.
	Otherwise it is open in $U(F)$ and simple by Theorem \ref{thmtitssimpl}.
	If $U(F)$ is non-discrete, it is easy to find two non-commuting elements in $U(F)^+$,
	so $U(F)^+$ is not an abelian group. 
	Therefore $D(\mathcal{N}_F) \cap U(F)^+$ is non-trivial and normal in $U(F)^+$. 
	Simplicitly of $U(F)^+$ now implies $U(F)^+ \leq D(\mathcal{N}_F)$ and as a conclusion $D(\mathcal{N}_F)$ is open.
	
	Obviously $D(V_F)$ is a normal subgroup of $D((\mathcal{N}_F) \cap V_F)$.
	By the third isomorphism theorem
	the homomorphism $V_F/D(V_F) \to V_F/(D(\mathcal{N}_F) \cap V_F)$ is surjective.
	The second isomorphism theorem implies
	$$V_F/(D(\mathcal{N}_F) \cap V_F) \cong V_F \cdot D(\mathcal{N}_F) / D(\mathcal{N}_F).$$
	Since $V_F$ is dense and $D(\mathcal{N}_F)$ is open in $\UF$
	we know $V_F \cdot D(\mathcal{N}_F)= \UF$ and the result follows.
	%
\end{proof}

The following corollary is immediate.

\begin{corollary}
	If $d$ is even and $F$ is transitive, then $\UF$ is simple.
\end{corollary}

\subsection{Normal subgroups}

We want to understand what normal subgroups $\UF$ can have.

\paragraph{Sign of an almost automorphism.} Let $T_1,T_2$ be finite complete subtrees of $\tree$.
Let $\varphi \colon \tree \setminus T_1 \to \tree \setminus T_2$ be a $U(F)$-honest almost automorphism.
If $\varphi \in V_F$ then by enlarging $T_1$ and $T_2$ if necessary we assume that
$\varphi$ is induced by the bijection
$\varphi|_{\mathcal{L}T_1} \colon \mathcal{L}T_1 \to \mathcal{L}T_2$.
Consider an $F$-invariant subset $D' \subset D$.
Then $\varphi$ induces a bijection
\[
\kappa \colon \{v \in \mathcal{L}T_1 \mid \ell_F(v) \subset D'\} \to \{v \in \mathcal{L}T_2 \mid \ell_F(v) \subset D'\}.
\]
Recall that in Section \ref{chht} we defined the lexographical order on the plane ordered tree $\mathcal{T}$.
There exists a unique order-preserving bijection
\[
\iota \colon \{v \in \mathcal{L}T_2 \mid \ell_F(v) \subset D'\} \to \{v \in \mathcal{L}T_1 \mid \ell_F(v) \subset D'\}.
\]
We define $\varphi_{\mathcal{L}_{D'}T_1} := \iota \circ \kappa$. 
Denote by $\sgn_{D'}(\varphi) \in \{1,-1\}$ the sign of the permutation $\varphi_{\mathcal{L}_{D'}T_1}$. 
Recall that, as we have seen in Section \ref{chht} for $D'=D$, it is only defined on honest almost automorphisms and ist not constant on equivalence classes in general.

\begin{proposition} \label{sigmawelldef}
	Let $D' \subset D$ be $F$-invariant.
	\begin{enumerate}
		\item[a)] The sign $\sgn_{D'}$ induces a well-defined homomorphism $V_F \to \{1,-1\}$ if and only if the cardinality $|D'|$ is even.
		\item[b)] The sign $\sgn_{D'}$ induces a well-defined homomorphism $\mathcal{N}_F \to \{1,-1\}$ if and only if the following two conditions are satisfied.
		\begin{enumerate}
			\item[1.] For every $\chi \in D$ holds
			$\{f|_{D'} \mid f \in F \colon f(\chi)=\chi\} \leq \operatorname{Alt}(D')$.
			\item[2.] The cardinality $|D'|$ is even.
		\end{enumerate}
	\end{enumerate}
\end{proposition}

\begin{proof}
	Let $T_1,T_2,T_3$ be finite complete subtrees of $\tree$ and let
	\begin{align*}
	\psi &\colon \tree \setminus T_1 \to \tree \setminus T_2 \\
	\psi' &\colon \tree \setminus T_2 \to \tree \setminus T_3
	\end{align*}
	be $U(F)$-honest almost automorphisms.
	It is clear that $$\sgn_{D'}(\psi' \circ \psi)=\sgn_{D'}(\psi')\sgn_{D'}(\psi).$$
	
	We now prove the ``if''-parts of a) and b).
	Consider a $U(F)$-honest almost automorphism $\varphi \colon \tree \setminus T_1 \to \tree \setminus T_2$.
	We need to show that for an equivalent honest almost automorphism $\varphi'$ with $T_1$ replaced by a simple expansion $T_1'$ (and similarly $T_2$ replaced by a simple expansion $T_2'$) holds $\sgn_{D'}(\varphi)=\sgn_{D'}(\varphi')$.
	Recall that an inversion of the permutation $\varphi_{\mathcal{L}_{D'}T_1}$ is a pair $(v,w)$
	such that $v<w$ but $\varphi_{\mathcal{L}_{D'}T_1}(v)>\varphi_{\mathcal{L}_{D'}T_1}(w)$.
	Also recall that the sign of a permutation is $1$ or $-1$ depending on if the number of its inversions is even or odd. Denote the set of inversions of a permutation $\rho$ by
	$\operatorname{Inv}(\rho)$.
	
	Let $w_0 \in \mathcal{L}T_1$ be the leaf of $T_1$ whose children are leaves of $T_1'$.
	%
	%
	
	\emph{Step 1: The ``if''-part of a).}
	Assume that $|D'|$ is even. Assume that the $U(F)$-honest almost automorphism $\varphi$ is the element of $V_F$ induced by the bijection $\varphi|_{\mathcal{L}T_1} \colon \mathcal{L}T_1 \to \mathcal{L}T_2$.
	
	Observe that
	\begin{align*}
	\operatorname{Inv}(\varphi'_{\mathcal{L}_{D'}T'_1})
	&= \{(v,w) \in \operatorname{Inv}(\varphi_{\mathcal{L}_{D'}T_1}) \mid v \neq w_0 \neq w\} \\
	& \quad \sqcup
	\{(v,w) \in \operatorname{Inv}(\varphi'_{\mathcal{L}_{D'}T'_1}) \mid v \text{ child of }w_0, \,
	w \text{ no child of }w_0\} \\
	& \quad \sqcup
	\{(v,w) \in \operatorname{Inv}(\varphi'_{\mathcal{L}_{D'}T'_1}) \mid v \text{ no child of }w_0, \,
	w \text{ child of }w_0\} \\
	& \quad \sqcup
	\{(v,w) \in \operatorname{Inv}(\varphi'_{\mathcal{L}_{D'}T'_1}) \mid v,w \text{ children of }w_0\}.
	\end{align*}
	The second assumption of Step 1 implies
	\[
	\{(v,w) \in \operatorname{Inv}(\varphi'_{\mathcal{L}_{D'}T'_1}) \mid v,w \text{ children of }w_0\} = \emptyset.
	\]
	Let $w$ be a child of $w_0$ and $v \in \mathcal{L}T_1 \setminus \{w_0\}$ such that $\ell_F(v),\ell_F(w)\subset D'$.
	Then $(v,w)$ is an inversion for $\varphi'_{\mathcal{L}_{D'}T'_1}$
	if and only if for every child $w'$ of $w_0$ with $\ell_F(w') \subset D'$ the pair $(v,w')$ is an inversion. The analogous statement holds for the pair $(w,v)$.
	Therefore the cardinalities of the second and third set in above union are divisible
	by the number of children of $w_0$ whose label is contained in $D'$.
	We now distinguish two cases.
	
	\emph{Case 1:} $\ell_F(w_0) \nsubseteq D'$
	
	In this case
	\[
	\{(v,w) \in \operatorname{Inv}(\varphi_{\mathcal{L}_{D'}T_1}) \mid v \neq w_0 \neq w\}
	= \operatorname{Inv}(\varphi_{\mathcal{L}_{D'}T_1}).
	\]
	Furthermore the number of children of $w_0$ whose label is a subset of $D'$ is $|D'|$,
	hence even by assumption.
	Therefore $|\operatorname{Inv}(\varphi_{\mathcal{L}_{D'}T_1})|$ is even if and only if
	$|\operatorname{Inv}(\varphi'_{\mathcal{L}_{D'}T'_1})|$ is even.
	Consequently $\sgn_{D'}(\varphi')=\sgn_{D'}(\varphi)$.
	
	\emph{Case 2:} $\ell_F(w_0) \subset D'$
	
	Let $v \neq w_0$ be a leaf of $T_1$ with $\ell_F(v) \subset D'$.
	Let $w$ be a child of $w_0$ with $\ell_F(w) \subset D'$.
	Note that by definition of the lexicographical order holds
	$(v,w) \in \operatorname{Inv}(\varphi'_{\mathcal{L}_{D'}T'_1})$ if and only if
	$(v,w_0) \in \operatorname{Inv}(\varphi_{\mathcal{L}_{D'}T_1})$.
	This implies
	\begin{align*}
	|\{(v,w) \in &\operatorname{Inv}(\varphi'_{\mathcal{L}_{D'}T'_1}) \mid v \text{ child of }w_0, \,
	w \text{ no child of }w_0\}| \\
	&= (|D'|-1) \cdot |\{(v,w) \in \operatorname{Inv}(\varphi_{\mathcal{L}_{D'}T_1}) \mid w=w_0 \}|
	\end{align*}
	and since, by assumption, the number $|D'|-1$ is odd, the cardinality
	$|\{(v,w) \in \operatorname{Inv}(\varphi'_{\mathcal{L}_{D'}T'_1}) \mid v \text{ child of }w_0, \,
	w \text{ no child of }w_0\}|$ is even if and only if
	$|\{(v,w) \in \operatorname{Inv}(\varphi_{\mathcal{L}_{D'}T_1}) \mid w=w_0 \}|$ is even.
	The analogous statement holds for $(w,v)$ instead of $(v,w)$.
	Consequently
	\begin{align*}
	|\operatorname{Inv}(\varphi'_{\mathcal{L}_{D'}T'_1})|
	&\equiv |\{(v,w) \in \operatorname{Inv}(\varphi_{\mathcal{L}_{D'}T_1}) \mid v \neq w_0 \neq w\}| \\
	& \quad +
	|\{(v,w) \in \operatorname{Inv}(\varphi_{\mathcal{L}_{D'}T_1}) \mid w=w_0 \}| \\
	& \quad + |\{(w,v) \in \operatorname{Inv}(\varphi_{\mathcal{L}_{D'}T_1}) \mid w=w_0 \}| \\
	&\equiv |\operatorname{Inv}(\varphi_{\mathcal{L}_{D'}T_1})| \mod (2).
	\end{align*}
	This implies $\sgn_{D'}(\varphi')=\sgn_{D'}(\varphi)$.
	%
	
	\emph{Step 2: The ``if''-part of b).} Assume that Assumptions 1. and 2. hold.
	
	Recall that the bijection $\varphi|_{\mathcal{L}T_1} \colon \mathcal{L}T_1 \to \mathcal{L}T_2$ uniquely determines an element of $V_F$. 
	By post-composing $\varphi$ with the inverse of this element and Step 1, we can assume $\varphi|_{\mathcal{L}T_1} = id$.
	Now passing to $\varphi'$, we see that the only possible inversions for $\varphi'_{\mathcal{L}_{D'}T'_1}$ are amongst the children of $w_0$ with label in $D'$.
	Note that they are permuted by an elmemet of $F$ that fixes $w_0$.
	By Assumption 2., there are evenly many inversions. This concludes the proof that $\sgn_{D'}(\varphi)=\sgn_{D'}(\varphi')$.
	
	For the ``only if''-parts denote by $S_n(v_0)$ the vertices of distance $n$ to $v_0$ and denote by $B_n(v_0)$ the finite complete subtree of $\tree$ spanned by $S_n(v_0)$.
	
	\emph{Step 3: The ``only if''-part of a).}
	Assume that $|D'|$ is odd.
	
	Let $\chi \in D'$.
	Let $n \geq 2$ be an integer and let $v,w \in S_n(v_0)$ be such that the colour of their parent edges is $\chi$.
	Let $\varphi \colon \tree \setminus B_n(v_0) \to \tree \setminus B_n(v_0)$  be the element of $V_F$ induced by the transposition of $v$ and $w$.
	Then $\varphi_{\mathcal{L}_{D'}B_{n}(v_0)}$ is a transposition
	and therefore $\sgn_{D'}(\varphi)=-1$.
	We consider now the $U(F)$-honest almost automorphism $\varphi' \colon \tree \setminus B_{n+1}(v_0) \to \tree \setminus B_{n+1}(v_0)$ that is equivalent to $\varphi$.
	The permutation $\varphi'_{\mathcal{L}_{D'}B_{n+1}(v_0)}$ is 
	the product of $|D'|-1$ many transpositions.
	Since $|D'|$ is odd, $\sgn_{D'}(\varphi')=1$.
	Therefore $\sgn_{D'}$ is not well-defined on equivalence classes of almost automorphisms.
	
	\emph{Step 4: The ``only if''-part of b).} Assume that Assumption 1. does not hold.
	
	Let $n \geq 1$.
	Let $f \in F$ and $\chi \in D$ be such that $f(\chi)=\chi$ and such that $f|_{D'} \notin \Alt(D')$. 
	Choose $g \in U(F)$ as follows. Pick a vertex $v \in S_n(v_0)$ such that the colour of its parent edge of $v$ is $\chi$. 
	Let $g|_{\tree \setminus \tree_v}=id$ and $\prm_{g,v}=f$. Note that this implies $g(v)=v$ and $g(\tree_v)=\tree_v$. 
	Let $\varphi$ and $\varphi'$ be $U(F)$-honest almost automorphisms
	$\varphi \colon \tree \setminus B_n(v_0) \to \tree \setminus B_n(v_0)$ 
	and $\varphi' \colon \tree \setminus B_{n+1}(v_0) \to \tree \setminus B_{n+1}(v_0)$ equivalent to $g$. 
	Then, since $f$ is an odd permutation we know that $\sgn_{D'}(\varphi) = 1$ and $\sgn_{D'}(\varphi') = -1$.
	Therefore $\sgn_{D'}$ is not well-defined on the equivalence classes of almost automorphisms.
\end{proof}

\begin{example}
	Let $d=6$ and let $F=\langle (1\,2)(3\,4),(5\,6) \rangle$.
	Then $F$ has four orbits, one of them with odd cardinality, and therefore $[V_F:D(V_F)]=8$.
	Consider $0 \in D$. Its stabilizer in $F$ equals $F$ and restricts to a subgroup of $\Alt(D')$ with $|D'|$ even
	for $D'=\{1,2,3,4\}$.
	The stabilizer of $1,2,3,4$ restricts to  a subgroup of $\Alt(D')$ with $|D'|$ even for $D'=\{1,2,3,4\}$.
	The stabilizer of $5$ and $6$ restricts to  a subgroup of $\Alt(D')$ with $|D'|$ even for $D'=\{1,2,3,4\}$ and $D'=\{1,2,3,4,5,6\}$.
	Therefore $\sgn_{D'}$ is a well-defined homomorphism only for $D'=\{1,2,3,4\}$.
\end{example}

\begin{remark}
	The set $\Delta := \{D' \subset D \mid D' \text{ as in Prop. \ref{sigmawelldef}b)} \}$ is closed under symmetric difference, so it is an abelian group where every element has order $2$.
	It seems plausible, and is true for $V_F$, that $\UF/D(\UF)$ is isomorphic to this group via an isomorphism induced by
	\[
	\UF \to \Delta, \quad \varphi \mapsto \sum_{\sgn_{D'}(\varphi)=-1} D'.
	\]
\end{remark}

\section{No lattices}

In this section $F \leq \Sym(D)$ will always be a Young subgroup with orbits $D^{(0)},\dots,D^{(l)} \subset D$. Recall that this means  $F=\prod_{i=0}^l \Sym(D^{(i)})$.
Denote again $\UF := \mathcal{N}(U(F))$.
The main goal of this section is to prove the  following theorem.

\begin{theorem} \label{nolattice}
	Assume $F \leq \Sym(D)$ is a Young subgroup with less than $d$ orbits.
	Then the group $\mathcal{N}_F$ does not admit a lattice.
\end{theorem}

The case $F=\Sym(D)$ is the content of \cite{bcgm12}.
Our proof follows the same argument.

\begin{notation}
	For each $i=0,\dots,l$ we write $d^{(i)} := |D^{(i)}|$.
	
	For $n \geq 0$ we denote by $S_n(v_0)$ the set of vertices of distance $n$ from $v_0$ and by $B_n(v_0)$ the subtree spanned by $S_n(v_0)$, i.e. the smallest subtree of $\tree$ containing $S_n(v_0)$.
	
	Denote by $O_n$ the equivalence classes of all $U(F)$-honest almost automorphisms 
	$\tree \setminus B_{n}(v_0) \to \tree \setminus B_{n}(v_0)$. 
	It is easy to see that $O_n < \UF$ is a compact and open subgroup and that $O_n < O_{n+1}$.
	For $n=0$ we have $B_0(v_0) = \{v_0\}$ and $O_0 = \operatorname{Fix}_{U(F)}(v_0) < U(F)$.
	Let
	\[
	O = \bigcup_{n \geq 0} O_n.
	\]
	
	Denote by $\mu$ the Haar measure on $\mathcal{N}_F$ normalized by $\mu(O_0)=1.$ 
	For $n \geq 0$ denote
	\[
	U_{n} = \operatorname{Fix}_{U(F)}(B_n(v_0)) < U(F).
	\]
	In particular $U_0=O_0$.
	The collection $\{U_n \mid n \geq 0\}$ is a neighbourhood basis of the identity for $U(F)$ and therefore also for $\mathcal{N}_F$ and $O$.
	
	For $n \geq 1$ the group $O_n$ acts on the $(d+1)d^{n-1}$ leaves of $B_{n}(v_0)$.
	Denote this action by $$\pi_n \colon O_n \to \Sym(S_n(v_0)).$$
	Its kernel is $U_n$. Clearly
	it has $l+1$ orbits $D_n^{(0)}, \dots, D_n^{(l)}.$
	We can determine $D_n^{(i)}$ explicitly, namely $$D_n^{(i)} = \{v \in S_n(v_0) \mid \ell_F(v) = D^{(i)}\},$$
	where as in the previous section $\ell_F(v)$ is the $F$-orbit of the parent edge of $v$.
	Since $O_n$ preserves the partition
	\smash{$S_n(v_0) = \bigsqcup_{i=0}^l D_n^{(i)}$},
	acts on each $D_n^{(i)}$ as the whole symmetric group 
	and for $i \neq j$ permutes the vertices of $D_n^{(i)}$ and $D_n^{(j)}$ independently, 
	the image of $\pi_n$ is $\prod_{i=0}^{l} \operatorname{Sym}(D_n^{(i)})$.
	Therefore $\pi_n$ induces an isomorphism
	\[
	O_n/U_n \cong \prod_{i=0}^{l} \operatorname{Sym}(D_n^{(i)}).
	\]
\end{notation}

\subsection{The group $O$ has no lattice}

Since $O < \mathcal{N}_F$ is open, the following theorem implies Theorem \ref{nolattice}.
\begin{theorem} \label{Onolattice}
	Assume $l < d-1$ and $F=\prod_{i=0}^l \Sym(D^{(i)})$. Then the group $O$ does not admit any lattice.
\end{theorem}

\begin{remark}[\emph{Strategy of the proof of Theorem \ref{Onolattice}}]
	Let by contradiction $\Gamma < O$ be a lattice.
	Denote its covolume by $c$. Similarly denote by $c_n$ the covolume of $\Gamma \cap O_n$ in $O_n$.
	We will now establish a lower bound for $c$ in terms of the index of $\Gamma_n := \pi_n(\Gamma \cap O_n)$ in \smash{$\prod_{i=0}^l \operatorname{Sym}(D_n^{(i)})$} and use it to get an estimate for the index
	$[\Sym(S_n(v_0)):\Gamma_n]$. Using this estimate we will, precisely as in \cite{bcgm12}, 
	find non-trivial elements in $\Gamma \cap U_n$ for very large $n$, which shows that $\Gamma$ cannot be discrete.
\end{remark}
\begin{notation}
	For a subgroup $G \leq \AutT$ we denote 
	\[\Aut_G(B_n(v_0)) = \{g \in \Aut(B_n(v_0)) \mid \exists h \in G\colon g=h|_{B_n(v_0)} \}.\]
\end{notation}

\paragraph{Covolume estimate.} Since $\Gamma$ is discrete there exists an $n_0 \in \mathbb{N}$ such that for all $n > n_0$ holds $\Gamma \cap U_n = \{1\}$.
That implies $\Gamma \cap O_n \cong \pi_n(\Gamma \cap O_n) =: \Gamma_n$ since $U_n = \ker(\pi_n)$.
We can make for $n \geq n_0$ the volume computation
\begin{align} \label{covolest}
\begin{split}
c \geq  c_n &= \operatorname{vol}(O_n/\Gamma \cap O_n) = \frac{\mu(O_n)}{|\Gamma \cap O_n|} = \frac{[O_n:O_0]}{|\Gamma_n|} = \frac{[O_n:U_n]}{|\Gamma_n|\cdot [U_0:U_n]} \\
&= \frac{|\prod_{i=0}^{l} \operatorname{Sym}(D_n^{(i)})|}{|\Gamma_n|\cdot |\operatorname{Aut}_{U(F)}(B_n(v_0))|}
= \frac{\big[\prod_{i=0}^{l} \operatorname{Sym}(D_n^{(i)}):\Gamma_n\big]}{|\operatorname{Aut}_{U(F)}(B_n(v_0))|}.
\end{split}
\end{align}


To prove that $\Gamma$ cannnot exist, we need a preparatory estimate.

\begin{proposition} \label{bigest}
	If $l<d-1$ and $d>2$ there exists a constant $C=C(d)$ such that for $n$ big enough holds
	\[
	[\operatorname{Sym}(S_n(v_0)) : \pi_n(\Gamma)] \leq C \cdot d^{|S_n(v_0)|}.
	\]
\end{proposition}

\begin{remark}
	One can check that if $l=d-1$ then the inequality is reversed.
	This corresponds to exactly $l$ of the numbers $d^{(i)}$ being equal to $1$ and the remaining one equal to $2$.
	
	The case $F=\Sym(D)$, in particular the case $d=2$ and $l=0$, is covered in \cite{bcgm12}.
\end{remark}

Before going to the quite technical proof of this Proposition we will derive Theorem \ref{Onolattice} from it.
We rephrase the key proposition from \cite{bcgm12}, which roughly says that subgroups of a huge finite symmetric group satisfying a certain index bound must contain one large alternating group or a product of many not so small alternating groups.

\begin{proposition}[\emph{Proposition 4.1 from \cite{bcgm12}}] \label{alternative}
	Let $c,d > 0$ be positive real numbers and $0<\alpha<1$. There exists an integer $n_1$ depending on $c,d$ and $\alpha$ such that for every finite set $K$ with $|K|\geq n_1$  every subgroup $\Lambda \leq \Sym(K)$ with $$[\Sym(K):\Lambda] \leq c \cdot d^{|K|}$$ satisfies one of the following (non-exclusive) alternatives.
	\begin{enumerate}
		\item There exists a subset $Z \subset K$ with $|Z| > \frac{|K|}{d}+2$ and $\operatorname{Alt}(Z) \leq \Lambda$.
		\item There exist $d$ disjoint subsets $Z_1,\dots,Z_d \subset K$ which satisfy
		$$\Big|\bigsqcup_{j=1}^d Z_j\Big| > (1-\alpha)|K| \quad \text{and} \quad \prod_{j=1}^d \operatorname{Alt}(Z_j) \leq \Lambda.$$
	\end{enumerate}
\end{proposition}

The conclusion of the proof that $O$ does not have a lattice works exactly as in \cite{bcgm12}. For completeness we reproduce it here.

\begin{proof}[Proof of Theorem \ref{Onolattice}]
	Let $\alpha < 1/d^2$. By Proposition \ref{bigest} we may apply Proposition \ref{alternative} to $K=S_n(v_0)$ and $\Lambda = \Gamma_n$ with $C,d$ and $\alpha$ for some fixed
	$n \geq \max\{n_0+2,n_1\}.$ Note that the choice of $n$ implies $\Gamma \cap U_{n-2} = \{1\}$.
	
	We introduce some terminology. Vertices with same parent are called siblings. Grandparents and grandchildren are defined in the obvious way.
	
	First assume that $\Gamma_n$ satisfies Alternative 1.
	Then, by the pigeonhole principle there need to exist either three siblings 
	$v_1,v_2,v_3$ or two pairs of siblings $w_1,w_2$ and $w_3,w_4$ in the set $Z$.
	Since $\Alt(Z) \leq \Gamma_n$, the
	corresponding permutation $(v_1 \, v_2 \, v_3) \in \Alt(Z)$ or
	$(w_1 \, w_2)(w_3 \, w_4) \in \Alt(Z)$ is in $\Gamma_n$.
	These permutations only permute amongst siblings, the preimage $\gamma$ under $\pi_n$ of this element must be a nontrivial element in $\Gamma \cap \Fix_{\AutT}(B_{n-1}(v_0))$ such that for all vertices $v$ outside of $B_n(v_0)$ the local permutation $\prm_{\gamma,v}$ is in $F$.
	Since $F$ is a Young subgroup, by Lemma 3.3 in \cite{lb16}, or alternatively by Proposition \ref{lemsameorb}, the element $\gamma$ is contained in $U_{n-1} \cap \Gamma$. This contradicts $n-1 \geq n_0$.
	
	Assume now that $\Gamma_n$ satisfies Alternative 2. We can assume in addition that $\Gamma_n$ does not contain a nontrivial element that only permutes siblings,
	because otherwise we get a contradiction as above.
	This means that every $Z_j$ contains at most one pair of siblings.
	We call siblings that are contained in the same $Z_j$ twins.
	Note that there are as many $Z_j$ as every parent has children.
	So, if a parent does not have twins, but still all its children
	are contained $Z:=\bigsqcup_{j=1}^d Z_j$,
	then every $Z_j$ contains exactly one of their children.
	Note that there are at most $d$ parents of twins and thus also at most $d$ grandparents of twins.
	
	There are at most $\alpha \cdot |S_{n}(v_0)|$ vertices in $S_{n-1}(v_0)$ having a grandchild that is not in $Z$. Since $\alpha < 1/d^2$ and $|S_{n}(v_0)|$, this means that at least $(1/d^2-\alpha)\cdot |S_{n}(v_0)|$ grandparents in $S_{n-1}(v_0)$ have all their grandchildren in $Z$.
	If $n$ is such that $(1/d^2-\alpha)\cdot |S_{n}(v_0)| \geq d+2$, there are at least two grandparents $g_1,g_2 \in S_{n-2}(v_0)$ all of whose grandchildren are in $Z$ but who are not grandparents of twins.
	For each of the two $g_i$, we can construct an element in $\prod_{j=1}^d \Sym(Z_j)$ by switching two of their children in a way that the grandchildren do not change $Z_j$ they are contained in. By composing these two elements, we get an element in $\prod_{j=1}^l \operatorname{Alt}(Z_j) \leq \Gamma_n$ whose preimage under $\pi_n$ lies in $U_{n-2}\cap \Gamma$, contradiction.
\end{proof}

In the proof of Proposition \ref{bigest} we will need the following formulae.

\begin{lemma} \label{lemexplcard}
	For $n \geq 1$ holds
	\begin{align*}
	\Big|\prod_{i=0}^l \operatorname{Sym}(D_n^{(i)})\Big| &= \prod_{i=0}^l (d^{(i)} \cdot d^{n-1})! \\
	|\operatorname{Aut}_{U(F)}(B_n(v_0))| &= \left( \prod_{i=0}^l d^{(i)}! \right) \cdot
	\left(\prod_{i=0}^{l}\frac{d^{(i)}!^{(d+1)}}{d^{(i)^{d^{(i)}}}}\right)^{\frac{d^{n-1}-1}{d-1}}.
	\end{align*}
\end{lemma}

\begin{proof}
	By symmetry, for every $\chi \in D$ the number of leaf edges $e$ of $B_n(v_0)$ with $\col(e)=\chi$ is 
	$\frac{|S_n(v_0)|}{|D|}=d^{n-1}$. This implies $d_n^{(i)}=d^{(i)} \cdot d^{n-1}$ and
	\[
	\Big|\prod_{i=0}^l \operatorname{Sym}(D_n^{(i)})\Big| = \prod_{i=0}^l (d^{(i)} \cdot d^{n-1})!.
	\]
	Restriction to $B_{n-1}(v_0)$ defines a surjective homomorphism $$\rho_n \colon \operatorname{Aut}_{U(F)}(B_n(v_0)) \twoheadrightarrow \operatorname{Aut}_{U(F)}(B_{n-1}(v_0)).$$
	The kernel of $\rho_n$ consists of all those automorphisms of $B_{n}(v_0)$ which 
	fix $B_{n-1}(v_0)$.
	
	Incident to each leaf of label $D^{(i)}$ in $B_{n-1}(v_0)$ are $d^{(j)}$ leaves of $B_{n}(v_0)$ with label $D^{(j)}$ for $j\neq i$ and $d^{(i)}-1$ leaves with label $D^{(i)}$.
	Thus we get $|\operatorname{Aut}_{U(F)}(B_{1}(v_0))|=\prod_{i=0}^l d^{(i)}!$ and for $n \geq 2$ we see
	\begin{align*}
	|\operatorname{Ker}(\rho_n)| &= \prod_{i=0}^l \prod_{D_{n-1}^{(i)}} (d^{(i)}-1)! \prod_{j \neq i} d^{(j)}!
	=  \left(\prod_{i=0}^{l}\frac{d^{(i)}!^{(d+1)}}{d^{(i)^{d^{(i)}}}}\right)^{d^{n-2}}.
	\end{align*}
	Inductively we get
	\begin{align*}
	|\operatorname{Aut}_{U(F)}(B_n(v_0))|  
	&= |\operatorname{Aut}_{U(F)}(B_{n-1}(v_0))| \cdot |\operatorname{Ker}(\rho_n)| \\
	&= |\operatorname{Aut}_{U(F)}(B_{1}(v_0))| \cdot \prod_{k=2}^n |\operatorname{Ker}(\rho_k)| \\
	&= |\operatorname{Aut}_{U(F)}(B_{1}(v_0))| \cdot \prod_{k=2}^n \left(\frac{\prod_{j=0}^{l} d^{(j)}!^{(d+1)}}{\prod_{i=0}^{l} d^{(i)^{d^{(i)}}}}\right)^{d^{k-2}} \\
	&= \left( \prod_{i=0}^l d^{(i)}! \right) \cdot
	\left(\prod_{i=0}^{l}\frac{d^{(i)}!^{(d+1)}}{d^{(i)^{d^{(i)}}}}\right)^{\frac{d^{n-1}-1}{d-1}}.
	\end{align*}
\end{proof}

\begin{proof}[Proof of Proposition \ref{bigest}]
	Recall Estimate (\ref{covolest}), namely
	\[
	\Big[\prod_{i=0}^l \operatorname{Sym}(D_n^{(i)}) : \pi_n(\Gamma)\Big] \leq c \cdot |\operatorname{Aut}_{U(F)}(B_n(v_0))|.
	\]
	It is equivalent to
	\[
	[\operatorname{Sym}(S_n(v_0)) : \pi_n(\Gamma)] \leq c \cdot |\hspace{-2pt}\operatorname{Aut}_{U(F)}(B_n(v_0))| \cdot \big[\operatorname{Sym}(S_n(v_0)) : \prod_{i=0}^l \operatorname{Sym}(D_i^{n})\big]
	\]
	and therefore it suffices to show that
	\begin{equation*}\label{inequtoprove}
	|\operatorname{Aut}_{U(F)}(B_n(v_0))| \cdot \big[\operatorname{Sym}(S_n(v_0)) : \prod_{i=0}^l \operatorname{Sym}(D_n^{(i)})\big] \leq C \cdot d^{|B_n(v_0)|}.
	\end{equation*}
	We write out this inequality explicitly in the values calculated 
	in Lemma \ref{lemexplcard} and use the $O$-Notation.
	The term $\prod_{i=0}^l d^{(i)}!$ is constant, so the inequality is equivalent to
	\[
	\left(\prod_{i=0}^{l}\frac{d^{(i)}!^{(d+1)}}{d^{(i)^{d^{(i)}}}}\right)^{\frac{d^{n-1}-1}{d-1}}
	\cdot \frac{((d+1)d^{n-1})!}{\prod_{i=0}^l d_n^{(i)}!} \leq O(1) \cdot d^{(d+1)d^{n-1}}
	\]
	and after taking the logarithm it is equivalent to
	\begin{align*}
	\frac{d^{n-1}-1}{d-1}& \sum_{i=0}^l \left( (d+1)\ln(d^{(i)}!) - d^{(i)}\ln(d^{(i)}) \right) 
	+ \ln\left(((d+1)d^{n-1})!\right) \\
	& \qquad \qquad \qquad \qquad \qquad \leq (d+1)d^{n-1} \ln(d) + \sum_{i=0}^l \ln(d_n^{(i)}!) + O(1).
	\end{align*}
	We need to deal with factorials of powers. For that we invoke
	Stirling's estimate
	\[
	\sqrt{2\pi} \cdot m^{m+\frac{1}{2}} e^{-m} \leq m! \leq \sqrt{e} \cdot m^{m+\frac{1}{2}} e^{-m}
	\]
	showing that
	\[
	\ln(m!)=\Big(m+\frac{1}{2}\Big) \ln(m) - m + O(1).
	\]
	It yields
	\begin{align*}
	\ln\left(((d+1)d^{n-1})!\right) &= \Big((d+1)d^{n-1}+\frac{1}{2}\Big) \ln((d+1)d^{n-1}) 
	\\ & \quad \quad - (d+1)d^{n-1} + O(1) \\
	&= n \cdot d^{n-1}\left((d+1)\ln(d)\right) \\
	& \quad \quad + d^{n-1}(d+1)(\ln(d+1)-\ln(d)-1) + O(n)
	\end{align*}
	and
	\begin{align*}
	\ln(d_n^{(i)}!) &= \ln\left( (d^{(i)}d^{n-1})! \right) \\
	&= \Big(d^{(i)}d^{n-1} + \frac{1}{2}\Big) \ln(d^{(i)}d^{n-1}) - d^{(i)}d^{n-1}
	+ O(1)\\
	&= n \cdot d^{n-1} \left( d^{(i)} \ln(d) \right) \\
	& \quad + d^{n-1} d^{(i)} \left( \ln(d^{(i)}) - \ln(d) -1 \right) + O(n).
	\end{align*}
	Thus we obtain
	\begin{align*}
	&\frac{d^{n-1}}{d-1} \sum_{i=0}^l \left( (d+1)\ln(d^{(i)}!) - d^{(i)}\ln(d^{(i)}) \right) \\
	&\quad \quad +  n \cdot d^{n-1}\left((d+1)\ln(d)\right) \\
	& \quad \quad+ d^{n-1}(d+1)(\ln(d+1)-\ln(d)-1)  + O(n)\\
	&\leq (d+1)d^{n-1} \ln(d)\\
	& \quad \quad+ \sum_{i=0}^l \left( n \cdot d^{n-1} \left( d^{(i)} \ln(d) \right) + d^{n-1} d^{(i)} \left( \ln(d^{(i)}) - \ln(d) -1 \right) \right).
	\end{align*}
	The dominant term $n d^{n-1}$ appears on both sides with the same coefficient $(d+1)\ln(d)$, so we can eliminate it and compare the coefficients of the newly dominating term, namely $d^{n-1}$.
	They are
	\[\frac{1}{d-1}\sum_{i=0}^l \left( (d+1)\ln(d^{(i)}!) - d^{(i)}\ln(d^{(i)}) \right) + (d+1)(\ln(d+1)-\ln(d)-1)\]
	on the left and
	\[(d+1) \ln(d) + \sum_{i=0}^l \left( d^{(i)} \left( \ln(d^{(i)}) - \ln(d) -1 \right) \right) 
	= \sum_{i=0}^l \left( d^{(i)} \ln(d^{(i)}) \right) - (d+1) \]
	on the right.
	To conclude the proof it suffices to show that this left dominant coefficient is strictly smaller than this right dominant coefficient, which is equivalent to
	\begin{align*}
	\frac{d+1}{d-1}\sum_{i=0}^l\big(\ln(d^{(i)}!)\big) + (d+1)(\ln(d+1)-\ln(d)) < \frac{d}{d-1}\sum_{i=0}^l d^{(i)} \ln(d^{(i)}),
	\end{align*}
	i.e.
	\[
	(d-1)(\ln(d+1) - \ln(d)) < \frac{d}{d+1}\sum_{i=0}^l d^{(i)}\ln(d^{(i)}) - \sum_{i=0}^l \ln(d^{(i)}!) .
	\]
	This inequality is precisely the content of the next lemma.
\end{proof}

\begin{lemma} \label{smallest}
	If $l < d-1$ and $d >2$, then
	\[
	(d-1)(\ln(d+1) - \ln(d)) < \frac{d}{d+1}\sum_{i=0}^l d^{(i)}\ln(d^{(i)}) - \sum_{i=0}^l \ln(d^{(i)}!) .
	\]
\end{lemma}

\begin{proof}
	For $d=2$ and $l=0$ we have equality. If $l=1$ and $(d^{(0)},d^{(1)})=(2,2)$ the inequality is true. Inductively and by symmetry in $d^{(0)},\dots,d^{(k)}$, the lemma will follow from two claims.
	
	For positive integers $x^{(0)},\dots,x^{(k)}$ with $\sum_{i=0}^k x^{(i)}=x+1$ and $k<x$ define the function $\Xi$ as the difference of the right hand side minus the left hand side, i.e.
	$$\Xi(x^{(0)},\dots,x^{(k)}) = \frac{x}{x+1}\sum_{i=0}^k x^{(i)}\ln(x^{(i)}) - \sum_{i=0}^k \ln(x^{(i)}!) + (x-1)\ln\left(\frac{x}{x+1}\right).$$
	
	\emph{Claim 1:} $\Xi(x^{(0)},\dots,x^{(k)},1) > \Xi(x^{(0)},\dots,x^{(k)})$
	
	\emph{Proof:} 
	Appending $1$ to the vector $(x^{(0)},\dots,x^{(k)})$ does not change the sums $\sum_{i=0}^k x^{(i)}\ln(x^{(i)})$ and $\sum_{i=0}^k \ln(x^{(i)}!)$, but increments $x$ by $1$. 
	After obvious simplifications and rearrangings of terms the desired inequality $\Xi(x^{(0)},\dots,x^{(k)},1) > \Xi(x^{(0)},\dots,x^{(k)})$ is equlivalent to
	\[
	\left(\frac{x+1}{x+2}-\frac{x}{x+1}\right)\sum_{i=0}^k x^{(i)}\ln(x^{(i)}) > x \ln\left(\frac{x+2}{x+1}\right) - (x-1) \ln\left(\frac{x+1}{x}\right).
	\]
	Because the function $(x^{(0)},\dots,x^{(k)})\mapsto \sum_{i=0}^k x^{(i)}\ln(x^{(i)})$ is convex and symmetric in $x^{(0)},\dots,x^{(k)}$, it attains its minimum if all the $x^{(i)}$ are the same, i.e.
	$$\sum_{i=0}^k x^{(i)}\ln(x^{(i)}) \geq (k+1)\frac{x+1}{k+1}\ln\left(\frac{x+1}{k+1}\right) \geq (x+1)\ln\left(\frac{x+1}{x-1}\right).$$
	We estimate the term dependent on the $x^{(i)}$ from below
	\begin{align*}
	\left(\frac{x+1}{x+2}-\frac{x}{x+1}\right)\sum_{i=0}^k x_i\ln(x_i) &\geq \left(\frac{x+1}{x+2}-\frac{x}{x+1}\right) (x+1)\ln\left(\frac{x+1}{x-1}\right) \\
	&= \frac{1}{x+2} \ln\left(\frac{x+1}{x-1}\right)
	\end{align*}
	and are left with showing that
	\[
	\xi(x):=\frac{1}{x+2} \ln\left(\frac{x+1}{x-1}\right) - x \ln\left(\frac{x+2}{x+1}\right) + (x-1) \ln\left(\frac{x+1}{x}\right) > 0.
	\]
	Oberve that
	\[
	\lim_{x \to \infty} \, x \ln\left(\frac{x+2}{x+1}\right) = \lim_{x \to \infty} \, \ln\left(\frac{\left(1+\frac{1}{x+1}\right)^{x+1}}{1+\frac{1}{x+1}}\right) = 1
	\]
	and similarly
	\[
	\lim_{x \to \infty} \, (x-1) \ln\left(\frac{x+1}{x}\right)= 1.
	\]
	Therefore $\xi(x)$ converges to $0$ as $x$ approaches infinity.
	The first three derivatives of $\xi$ are
	\begin{align*}
	\xi'(x) &=  -\frac{\ln \left(\frac{x+1}{x-1}\right)}{(x+2)^2}-\ln \left(\frac{x}{x+1}\right)-\frac{x^2-x+2}{x (x+2) \left(x^2-1\right)}
	-\ln \left(\frac{x+2}{x+1}\right)\\
	\xi''(x) &= -\frac{x^4-10x^3-15x^2+8x+4}{x^2 (x+2)^2 \left(x^2-1\right)^2}+2 \cdot \frac{\ln \left(\frac{x+1}{x-1}\right)}{(x+2)^3} \\
	\xi'''(x) &= -2 \cdot \frac{23 x^5+25 x^4-25 x^3-9 x^2+6 x+4}{x^3 (x+2)^2 \left(x^2-1\right)^3}-6 \cdot \frac{\ln \left(\frac{x+1}{x-1}\right)}{(x+2)^4}.
	\end{align*}
	Since $\xi'''$ is strictly negative for $x \geq 2$, we know that $\xi'$ is strictly concave. 
	In addition $\xi'$ converges to zero, so it must be negative.
	This implies that $\xi$ is a strictly decreasing function converging to zero.
	Therefore $\xi$ must be positive and Claim 1 follows.
	
	\emph{Claim 2:} $\Xi(x^{(0)},\dots,x^{(k-1)},x^{(k)}+1) > \Xi(x^{(0)},\dots,x^{(k)},1)$
	
	\emph{Proof:} The inequality $\Xi(x^{(0)},\dots,x^{(k-1)},x^{(k)}+1) > \Xi(x^{(0)},\dots,x^{(k)},1)$ is, after obvious simplifications, equivalent to $$\frac{x}{x+1}(x^{(k)}+1)\ln(x^{(k)}+1) - \ln((x^{(k)}+1)!) > \frac{x}{x+1} x^{(k)} \ln(x^{(k)}) - \ln(x^{(k)}!),$$
	which after exponentiating is equivalent to
	\[
	\left( \frac{(x^{(k)}+1)^{x^{(k)}+1}}{x^{(k)^{x^{(k)}}}} \right)^{\frac{x}{x+1}} > x^{(k)}+1.
	\]
	We estimate the left hand side from above by setting $x=x^{(k)}$ and get
	\[
	\left( 1+ \frac{1}{x}\right)^{x^2} > x+1
	\]
	which is true for $x\geq 2$ because $\left( 1+ \frac{1}{x}\right)^{x} > 2$ and $2^x > x+1$, so Claim 2 follows.
	
	To conclude the lemma from these two claims, observe that any vector $(d^{(0)},\dots,d^{(l)})$ as in the lemma arises from $(3)$ or $(2,2)$ by a sequence of operations as in Claim 1 and Claim 2 and rearranging coordinates.
\end{proof}

\subsection{The group $\mathcal{N}_{\langle (0 \, 1) \rangle}$ has no cocompact lattice}

We do not know if for a Young group $F$ with $d$ orbits the group $\UF$ has a non-cocompact lattice or not,
but at least we can prove the non-existence of cocompact lattices.
After conjugating with an element of $\AutT$ we can assume $F=\langle (0 \, 1) \rangle$.

\begin{theorem} \label{thmnococpctlattice}
	The group $\mathcal{N}_{\langle (0 \, 1) \rangle}$ has no cocompact lattice.
\end{theorem}

The proof works again as in \cite{bcgm12}. We need three lemmata.

\begin{lemma}[Ramanujan,\cite{r19}]\label{lemrama}
	For every $m \geq 17$ there exist three different prime numbers in the interval $(\frac{m}{2},m]$.
\end{lemma}

\begin{definition}
	Let $K$ be a finite set.
	A subgroup of $\Sym(K)$ is called \emph{primitive} if the only partitions of $K$ it preserves are the trivial partition $\{K\}$ and the atomic partition $\{\{k\} \mid k \in K\}$.
\end{definition}

\begin{lemma}[\cite{bcgm12}, Lemma 3.1.]\label{primcycleslem}
	A subgroup of $\Sym(K)$ generated by two prime cycles whose respective supports intersect nontrivially, but are not contained in one another, acts doubly transitively (in particular, primitively) on its support.
\end{lemma}

\begin{lemma}[Jordan's Theorem, see \cite{w64}, Theorem 13.9.]\label{lemjordan}
	A primitive subgroup of $\Sym(K)$ containing a $p$-cycle for a prime number $p \leq |K|-3$ is equal to $\Alt(K)$ or $\Sym(K)$.
\end{lemma}

\begin{proof}[Proof of Theorem \ref{thmnococpctlattice}]
	We again show that already $O$ does not have a cocompact lattice.
	Let, by contradiction, $\Gamma < O$ be a cocompact lattice.
	Consider now a compact fundamental domain of $\Gamma$. 
	Because $O$ is the increasing union of the $O_n$, for $n$ big enough, 
	the fundamental domain is contained in $O_n$. 
	Thus, for $n$ big enough, the sequence $(c_n)$ becomes constant and all inequalities in Estimate (\ref{covolest}) are actually equalities.
	This shows that $c$ is rational and for big $n$ holds
	\begin{align*}
	c &= \frac{|\prod_{i=0}^l \Sym(D_n^{(i)})|}{|\Aut_{U(F)}(B_n(v_0))| \cdot |\Gamma_n|} \\
	&= \frac{ (2 \cdot d^{n-1})!(d^{n-1}!)^{d-1}}{2^{d^{n-1}} \cdot |\Gamma_n|}.
	\end{align*}
	Observe that $(2 \cdot d^{n-1})!(d^{n-1}!)^{d-1}$ has arbitrarily big odd prime factors.
	All of them need to be cancelled out in the fraction above by $|\Gamma_n|$.
	By Lemma \ref{lemrama} there exist three different prime numbers in the closed interval $[d^{n-1}+1,2 \cdot d^{n-1}]$.
	Hence there exist primes $p,q$ such that $p+3<q$. None of their squares divides $(2 \cdot d^{n-1})!(d^{n-1}!)^{d-1}$.
	Consequently $|\Gamma_n|$ needs to
	be divisible by $p$ and by $q$, so by Cauchy's Theorem
	$\Gamma_n \leq \Sym(S_n(v_0))$ contains a $p$-cycle and a $q$-cycle.
	Without loss of generality we can assume $\{0,1\}=D^{(0)}$.
	Since $p,q > d_n^{(i)}$ for every $i\geq 1$, the mentioned $p$-cycle and $q$-cycle must be contained in
	$\Sym(D_n^{(0)})$ and intersect nontrivially. Now conjugating the $p$-cycle with the $q$-cycle
	we can produce another $p$-cycle whose support intersects the support of the original $p$-cycle non-trivially
	and such that the union of their supports has cardinality at least $p+3$.
	By Lemma \ref{primcycleslem} and Lemma \ref{lemjordan} we deduce that $\Gamma_n$ 
	contains the alternating group of some set of vertices $V_n^{(0)} \subset D_n^{(0)}$ of size $k>d_n^{(0)}/2+2$.
	
	Now as in the proof of Theorem \ref{Onolattice}, by the pigeonhole principle
	$V_n^{(0)}$ contains two pairs of siblings $(v_1,w_1)$ and $(v_2,w_2)$ such that
	$v_i \neq w_i$ for $i=1,2$ and $|\{v_1,v_2,w_1,w_3\}| \geq 3$. 
	The permutation $$\bar \gamma = (v_1,w_1)(v_2,w_2)$$ is an element of $\operatorname{Alt}(V_n^{(0)}) \subset \Gamma_n$, but 
	its pre-image $\gamma \in \Gamma \cap O_n$ is a non-trivial element of $U_{n-1}$.
	This is not possible for large $n$ and makes the existence of $\Gamma$ impossible.
\end{proof}

 \bibliographystyle{alpha}
 \bibliography{references}
\end{document}